\documentclass[a4paper,11pt,reqno]{article}
\usepackage[body=17cm,top=3cm,bottom=4cm]{geometry}
\usepackage[english]{babel}

\usepackage{amsrefs}
\usepackage{amsmath,amsfonts,amsthm,amssymb,mathrsfs,graphicx,color,dsfont,xcolor,bbm}
\usepackage{hyperref}
\usepackage{lipsum}
\usepackage{tikz}
\usepackage{geometry}
\usetikzlibrary{arrows,chains,matrix,positioning,scopes}
\usepackage{graphics}

\usepackage{cases}
\usepackage{bbm}
\usepackage{graphicx}
\usepackage{subfig}
\usepackage{float}
\usepackage{mathrsfs}
\usepackage{enumerate}
\usepackage{appendix}
\usepackage{times}
\usepackage{microtype}
\usepackage{mathrsfs}
\usepackage{tikz}
\usepackage{geometry}

\newcommand{\CE}{\mathcal{E}}

\DeclareMathOperator{\var}{Var}

\newcommand{\eps}{\varepsilon}

\makeatletter
\newcommand{\eqcolon}{\mathrel{\mathord{=}\raise.2\p@\hbox{:}}}
\newcommand{\coloneq}{\mathrel{\raise.2\p@\hbox{:}\mathord{=}}}
\makeatother
\newcommand{\mathD}{\mathrm{D}}
\newcommand{\mathd}{\mathrm{d}}

\newcommand{\nosymbol}{}

\newcommand{\tmop}[1]{\ensuremath{\operatorname{#1}}}
\newcommand{\tmtextit}[1]{{\itshape{#1}}}
\newcommand{\dd}{\mathrm{d}}

\newcommand{\para}{\,\mathord{\prec}\,}
\newcommand{\lpara}{\,\mathord{\succ}\,}
\newcommand{\mpara}{\,\mathord{\prec\!\!\!\prec}\,}
\newcommand{\reso}{\,\mathord{\circ}\,}

\newcommand{\CC}{\mathscr{C}}

\newcommand{\CA}{\mathscr{A}}
\newcommand{\CB}{\mathscr{B}}
\newcommand{\CD}{\mathscr{D}}
\newcommand{\CS}{\mathscr{S}}
\newcommand{\CF}{\mathscr{F}}
\newcommand{\DD}{\mathscr{D}}
\newcommand{\LL}{\mathscr{L}}

\newcommand{\PC}{\mathcal{P}}

\newcommand{\E}{\mathbb{E}}
\renewcommand{\P}{\mathbb{P}}
\newcommand{\Q}{\mathbb{Q}}
\newcommand{\R}{\mathbb{R}}
\newcommand{\Z}{\mathbb{Z}}
\newcommand{\T}{\mathbb{T}}
\newcommand{\C}{\mathbb{C}}
\newcommand{\N}{\mathbb{N}}
\newcommand{\1}{\mathds{1}}
\renewcommand{\textmu}{\mu}

\theoremstyle{plain}

\newtheorem{theorem}{Theorem}[section]

\newtheorem{corollary}[theorem]{Corollary}

\newtheorem{definition}[theorem]{Definition}

\newtheorem{lemma}[theorem]{Lemma}

\newtheorem{proposition}[theorem]{Proposition}

\newtheorem{remark}[theorem]{Remark}

\begin{document}

\title{An invariance principle for the two-dimensional \\ parabolic Anderson
model with small potential\thanks{We are very grateful to Massimiliano Gubinelli for countless discussions on the subject matter and in particular for suggesting the random operator approach that helped resolve the technically most challenging problem of the paper. We would also like to thank Hao Shen and Hendrik Weber for helpful discussions and suggestions.}}

\author{
  Khalil Chouk\\
  Humboldt-Universit{\"a}t zu Berlin \\
  Institut f\"ur Mathematik
  \and
  Jan Gairing \\
  Humboldt-Universit{\"a}t zu Berlin \\
  Institut f\"ur Mathematik
  \and
  Nicolas Perkowski\thanks{Financial support by the DFG via Research Unit FOR 2402 is gratefully acknowledged.} \\
  Humboldt-Universit{\"a}t zu Berlin \\
  Institut f\"ur Mathematik \\
  \texttt{perkowsk@math.hu-berlin.de}
}

\maketitle

\begin{abstract}
We prove an invariance principle for the two-dimensional lattice parabolic
 Anderson model with small potential. As applications we deduce a Donsker type convergence result for a discrete random polymer measure, as well as a universality result for the spectrum of discrete random Schr\"odinger operators on large boxes with small potentials. Our proof is based on paracontrolled distributions and some basic results for multiple stochastic integrals of discrete martingales.
\end{abstract}

\section{Introduction}

The discrete parabolic Anderson model (PAM) is the infinite-dimensional random ODE
\begin{equation}\label{eq:discrete pam intro}
   \partial_t v (t, i) = \Delta v (t, i) +  v (t, i) \eta (i),\quad (t,i)\in[0,+\infty)\times \Z^d,
\end{equation}
where $\Delta$ is the discrete Laplacian and $(\eta (i) : i \in \Z^d)$ is an i.i.d. family of random variables with sufficiently many moments. The discrete PAM has been intensely studied in the past decades due to the fact that it is the simplest known model that exhibits \emph{intermittency,} meaning roughly speaking that the bulk of the mass of the solution is concentrated in a few isolated islands. By now the intermittency properties of the discrete PAM are well understood, and it is known that the solution is intermittent whenever the $\eta(i)$ are truly random, even if they are bounded; see the surveys~\cite{CarmonaMolchanov1994} and~\cite{Koenig2016}. To get a better intuitive understanding of the PAM let us note that it models a branching random walk in random environment: Place independent particles on the lattice $\Z^d$ which all follow the dynamics of a continuous-time simple random walk, independently of $\eta$, and which at the lattice point $i$ get killed with rate $\eta(i)^-$ and branch into two new particles with rate $\eta(i)^+$; after the branching the two particles follow the same dynamics, independently of each other and all other particles. Then $v(t,i)$ is the expected number of particles at time $t$ in location $i$, conditionally on the random environment $\eta$. From this description it is intuitively convincing that $v$ should have high peaks in the regions where the environment is most favorable for the particles, and it should have deep valleys in between.

Therefore, we cannot expect to see a nontrivial behavior on large spatial scales. However, if we tune down the strength of the potential $\eta$ by considering
\begin{equation}\label{eq:discrete pam weak potential}
   \partial_t v (t, i) = \Delta v (t, i) + \varepsilon^{2-d/2} v (t, i) \eta (i),\quad (t,i)\in[0,+\infty)\times \Z^d,
\end{equation}
then for $d \le 3$ there is some hope to obtain a meaningful limit under the scaling $(t,x) \to (\varepsilon^{-2} t, \varepsilon^{-1} x)$ as long as $\eta(0)$ has $d/(2-d/2)+\delta$ moments. Indeed, on a time scale of length $\varepsilon^{-2}$ the simple random walk typically explores a region of size $\varepsilon^{-1}$, and we have 
\begin{align*}
   \lim_{\varepsilon \to 0} \E\left[ \max_{i \in (-\varepsilon^{-1},\varepsilon^{-1})^d} |\varepsilon^{2-d/2} \eta(i)|^{d/(2-d/2)+\delta} \right] & \le \lim_{\varepsilon \to 0} \sum_{i \in (-\varepsilon^{-1},\varepsilon^{-1})^d} \E[|\varepsilon^{2-d/2} \eta(i)|^{d/(2-d/2)+\delta}] \\
   & \lesssim \lim_{\varepsilon \to 0} \varepsilon^{\delta(2-d/2)} \E[|\eta(0)|^{d/(2-d/2)+\delta}] = 0,
\end{align*}
so that the influence of the potential felt by a typical particle converges to zero. Hence, we may hope that the intermittency properties of the solution do not dominate and there is a meaningful scaling limit. And indeed a formal computation suggests that if the $\eta(i)$ are centered (which can be always achieved by performing the change of variables $v(t) \to e^{-t \E[\eta(0)]} v(t)$), then in dimensions $d=1,2,3$ the rescaled solution $v(\varepsilon^{-2}t, \varepsilon^{-1}x)$ of~\eqref{eq:discrete pam weak potential} should converge to the solution $w$ of
\begin{equation}\label{eq:continuous pam intro no-renorm}
   \partial_t w (t, x) = \Delta w (t, x) + \sigma w (t, x) \xi (x), \hspace{2em} (t,x) \in \R_+ \times \R^d,
\end{equation}
where $\sigma^2 = \var(\eta(0))$ and $\xi$ is a space white noise, that is the centered Gaussian process with covariance $\E[\xi(x) \xi(y)] = \delta(x-y)$. Since we conjecture $w$ to be intermittent, proving this convergence would be the first step towards showing that the parabolic Anderson model is intermittent on the temporal scale $\varepsilon^{-2}$ and the spatial scale $\varepsilon^{-1}$ whenever the potential has at least strength $\varepsilon^{2-d/2}$, but not if it is weaker than that.

\medskip

Here we study the convergence of~\eqref{eq:discrete pam intro} to~\eqref{eq:continuous pam intro no-renorm}. We focus on the case $d=2$ and we consider the periodic model on $\Z_N^2 := (\Z/(N\Z))^2$ for $N \simeq \varepsilon^{-1}$. As we will see the naive derivation of~\eqref{eq:continuous pam intro no-renorm} does not give the full picture and there are more subtle effects to take into account: In dimensions $d=2,3$ the total number of particles grows exponentially fast and we have to look at the solution in a different scale to see a non-trivial behavior. More precisely, for
$t > 0$ the expected number of particles at time $\varepsilon^{-2} t$ will be of order $e^{t c_\varepsilon}$ with
$c_\varepsilon \simeq |\log \varepsilon|$ in $d=2$, so that we should instead consider $u_\varepsilon(t,x) :=
e^{- t c_\varepsilon} v(\varepsilon^{-2}t, \varepsilon^{-1}x)$ 
 which solves the modified equation
\begin{equation}\label{eq:scaled lattice pam renorm}
   \partial_t u_\varepsilon(t,x) = \Delta^\varepsilon_{\mathrm{per}} u_\varepsilon(t,x) + u_\varepsilon(t,x) (\eta_\varepsilon(x) - c_\varepsilon).
\end{equation}
Here $\Delta_{\mathrm{per}}^\varepsilon$ is the periodic discrete Laplacian, rescaled in such a way that it converges to the continuous periodic Laplace operator and $\eta_\varepsilon$ is a rescaled version of $\eta$ that converges to the white noise if we let $\varepsilon \to 0$.
This blow-up of the number of particles coincides nicely with the fact that the continuous equation~\eqref{eq:continuous pam intro no-renorm} only makes sense if a renormalization procedure is introduced. It was shown using regularity structures in~\cites{Hairer2014Regularity} and paracontrolled distributions in~\cite{Gubinelli2013} that if $\xi_\delta$ is a mollification of the white noise, then there exist diverging constants $(c_\delta)_{\delta>0}$ such that the solution $h_\delta$ of  
\begin{equation}\label{eq:continuous pam2}
   \partial_t h_\delta (t, x) = \Delta h_\delta (t, x) + h_\delta (t, x) (\xi_{\delta} (x)-c_{\delta})
\end{equation}
converges for $\delta \to 0$ to a nontrivial limit $u$ which solves an abstract equation of the form 
\begin{equation}\label{eq:continuous pam intro}
\partial_tu(t,x)=\Delta u(t,x)+u(t,x)\diamond\xi(x)
\end{equation}
with the renormalized product being formally given by $u\diamond\xi=u(\xi-\infty)$. Moreover, $u$ does not depend on the specific mollification of the white noise, and we have
\begin{equation}\label{eq:c-delta}
c_\delta=\frac{1}{2\pi}\log(\frac{1}{\delta})+O(1),
\end{equation}
where only the finite part of $c_\delta$ depends on the mollifier. 

\medskip

Our first result is that the solution $u_\varepsilon$ to~\eqref{eq:scaled lattice pam renorm} converges weakly to the solution $u$ of~\eqref{eq:continuous pam intro}, see Theorem~\ref{thm:main} below for a precise formulation. We are able to considerably weaken the assumptions on the potential $\eta$ and only require that it is given by appropriate martingale increments, and also we can replace the discrete Laplacian by the generator of any symmetric random walk whose increments have sufficiently many moments. The proof is based on paracontrolled distributions and a certain random operator technique developed in~\cite{Gubinelli2014KPZ}. The main technical contribution of this paper is to introduce suitable martingale tools in this context, which allow to control sufficiently many moments of the potential and some nonlinear functionals constructed from it, moment bounds which are needed as input for the paracontrolled machinery.

As a corollary of our convergence result we show a Donsker-type invariance principle for a certain random polymer measure, given by
\[
   \tilde{\Q}^\varepsilon_{T,x}(\dd \omega) = Z^{-1}_{\varepsilon,T,x} \exp\left(\int_0^{\varepsilon^{-2}T} \varepsilon \eta(\omega(s)) \dd s\right) \tilde{\P}^\varepsilon_x(\dd \omega),
\]
where $\tilde{\P}^\varepsilon_x$ is the law of a continuous-time random walk as above, started in $x$, and $Z_{\varepsilon,T,x}$ is a renormalization constant. We show in Theorem~\ref{thm:polymer} that the law of $(\varepsilon B^N_{\varepsilon^{-2} t})_{t\in [0,T]}$ under $\tilde{\Q}^\varepsilon_{T,x}$ converges to the continuum polymer measure which was recently constructed in~\cite{CannizzaroChouk2015}, a result which is universal for all appropriate random walk dynamics and laws of potentials.

Another simple consequence of Theorem~\ref{thm:main} is a universality result for the spectrum of the Anderson Hamiltonian on a large box with a small potential. Consider the operator $\mathscr H_\varepsilon$ on $\Z_N^2$ given by
\[
\mathscr H_\varepsilon v =-\Delta_{\mathrm{rw}}v+\eps v \eta,
\]
where $\Delta_{\mathrm{rw}}$ is the generator of a symmetric random walk with sufficiently many moments. We are interested in the behavior of the $k$ smallest eigenvalues $\Lambda_1^\varepsilon \le \dots \le \Lambda^\varepsilon_k$ of $\mathscr H_\varepsilon$, where $k$ is fixed and $N \to \infty$ (and thus $\varepsilon \simeq N^{-1} \to 0$). If we had $\eta \equiv 0$, then under the scaling $\varepsilon^{-2}(\Lambda_1^\varepsilon, \dots, \Lambda_k^\varepsilon)$ the eigenvalues would converge to the eigenvalues of the periodic Laplacian $-\Delta$, given by $0, 1, 1, 2, 2, 3, 3, \dots$. On the other side the minimal eigenvalue of the operator $v \mapsto \varepsilon v \eta$ clearly diverges to $- \infty$ when multiplied with $\varepsilon^{-2}$ because it is simply the minimum of $\varepsilon^{-1}\eta(i)$, $i \in \Z_N^2$. So one might guess and it turns out to be true that the bottom of the spectrum of $\mathscr H_\varepsilon v$ diverges to $- \infty$ when rescaled by a factor $\varepsilon^{-2}$. But what we are able to show is that a small logarithmic shift results in a nontrivial universal limit. More precisely, we prove in Theorem~\ref{thm:schroedinger} that for $c_\varepsilon \simeq |\log \varepsilon|$ as above we have
\[
   \varepsilon^{-2}\{(\Lambda_1^N,\dots,\Lambda_k^N)+ \varepsilon^2 c_\varepsilon(1,\dots,1)\}  \Rightarrow (\Lambda_1,,\dots,\Lambda_k)
\]
in distribution, where $(\Lambda_1,,\dots,\Lambda_k)$ are the $k$ minimal eigenvalues of the continuous Anderson Hamiltonian on the two-dimensional torus which was recently constructed in~\cite{AllezChouk2015}.

\medskip
The need for renormalization is a general feature of \emph{singular SPDEs} of which the 2d continuum PAM with white noise potential is the simplest example (in fact one can transform it into a well-posed equation by a change of variables~\cite{HairerLabbe}, but we will not make use of this). In recent years, following the fundamental work of Hairer~\cite{Hairer2011Rough, hairer_solving_2011}, there has been a breakthrough in the understanding of such equations which also include for example the $\Phi^4_d$ model in dimensions $d=2,3$~\cite{Hairer2014Regularity, CatellierChouk2013, Kupiainen2015}, the KPZ equation~\cite{hairer_solving_2011, Friz2014, Gubinelli2014KPZ, Hoshino2016} and its generalizations~\cite{Hoshino 2015, Kupiainen2016, Hairer2016, Bruned2015}, and the sine-Gordon equation~\cite{Hairer2015Sine}. The now available theories (regularity structures, paracontrolled distributions, and Kupiainen's renormalization group approach~\cite{Kupiainen2015}) all give the continuous dependence of the solution on some extended input, consisting of certain multilinear functionals constructed from the noise. So a priori they are well suited for proving the convergence of microscopic models to singular SPDEs. The main difficulty is that the theories are tailored for equations on Euclidean space (see however~\cite{Bailleul2016, Bailleul2015}), so some work is necessary to apply them to lattice systems such as the discrete PAM. Here we avoid this problem by finding a suitable extension of our lattice function to the continuous torus for which we can still write down a closed equation, a trick that was successfully used before in many works studying approximations of singular SPDEs~\cite{Hairer_Maas_2012, Hairer_Maas_2014, Mourrat2014, Gubinelli2014KPZ, Zhu2015, Zhu2014Discretization, Shen2016}. Alternatively, it would be possible to work with the lattice version of regularity structures that was developed by Hairer and Matetski in~\cite{HairerMatetski2015}. Once we are in a setting where we can apply one of the available theories for singular SPDEs, the next problem is how to control the multilinear functionals of the noise which are needed as input for the equation, and how to bound their moments to a sufficiently high order. In the Gaussian setting all moments are comparable and therefore it suffices to estimate the variance. However, even estimating the variance in a Gaussian setting can be tricky and over the past years Hairer and coauthors have made tremendous progress on finding efficient ways of doing so~\cite{Hairer2014Regularity, HairerPardoux, HairerQuastel, HairerXu, Bruned2015, Hairer2016}. In the non-Gaussian setting additional arguments are necessary, and different ways of tackling this problem have been developed in~\cite{Hairer2015Sine, HairerShenKPZ, ShenXu, ChandraShen}. Here we use an approach that is more closely related the one Mourrat and Weber used in~\cite{Mourrat2014} and Shen and Weber used in~\cite{Shen2016}. That is, we rely on martingale arguments and decompose the bilinear functional to be controlled in a sum of multiple stochastic integrals.

\medskip

The structure of the paper is as follows. In Section~\ref{sec:set up} we introduce our assumptions, state the convergence result for the discrete PAM, and show how to transform the lattice equation into a continuous PDE. Section~\ref{sec:paracontrolled} contains a short introduction to paracontrolled distributions and we briefly discuss the paracontrolled analysis of the continuous PAM before proceeding to use the paracontrolled tools to also control the continuous PDE derived from the lattice system. Here we obtain a pathwise convergence result under the assumption that some bilinear functionals constructed from the potential $\eta$ converge in the right topology. In Section~\ref{sec:martingales} we use discrete multiple stochastic integrals in order to prove the convergence of these bilinear functionals. Section~\ref{sec:polymer} contains the application to the polymer measure, and Section~\ref{sec:schroedinger} to the spectrum of the Anderson Hamiltonian.

\section{Mathematical set up}\label{sec:set up}

To rigorously state our convergence result we first have to introduce the required assumptions. We start by introducing two conditions on $N$:
\[
   \text{we have } N = \frac{2\pi}{\varepsilon} \qquad \text{and} \qquad N \text{ is odd.}
\]
Of course, we only assume $N$ to be odd for convenience since it simplifies the notation. Furthermore, we make the following assumptions on $\Delta_{\mathrm{rw}}$ and $\eta$:

\begin{description}
  \item[(H$_{\mathrm{rw}}$)] We have
\begin{equation}\label{eq:discrete operators def}
   \Delta_{\mathrm{rw}} \varphi (i) = \int_{\Z^2} \varphi (i + j) {\textmu} (\mathd j),
\end{equation}
where $\mu$ is a finite signed measure on $\Z^2$ with $\mu(\{j \}) \ge 0$ for all $j \neq 0$, and with $\int_{\Z^2} \mu(\dd j) = \int_{\Z^2} j_1 \mu(\dd j) =  \int_{\Z^2} j_2 \mu(\dd j) =  \int_{\Z^2} j_1 j_2 \mu(\dd j) = 0$, with $\int_{\Z^2} j_1^2 \mu(\dd j) = \int_{\Z^2} j_2^2 \mu(\dd j) = 2$ and with finite sixth moment. We also require $\mu$ to be radial (i.e. $j \mapsto \mu(\{j\})$ is a radial function) and that $\mu(\{(0,1)\}) > 0$.
  \item[(H$_{\mathrm{mart}}$)] There exists an enumeration $\zeta\colon \{0,\dots, N^2-1\} \to \Z_N^2$ of $\Z_N^2$ such that $(\eta_N(\zeta(k))_{k \le N^2-1}$ is a family of martingale differences (in its own filtration). Moreover, there exists $M>0$ such that for all $N \in \N$ and all $k \in \{0,\dots, N^2-1\}$
  \[
     \E[|\eta_N(\zeta(k))|^2|\eta_N(\zeta(0)), \dots, \eta_N(\zeta(k-1))] = 1, \quad  \E[|\eta_N(k)|^p |\eta_N(\zeta(0)), \dots, \eta_N(\zeta(k-1))] \le M
  \]
  for some $p > 6$.
\end{description}
Note that (H$_{\mathrm{rw}}$) is satisfied if $\mu$ corresponds to the transition rates of two independent symmetric random walks on $\Z$ that are combined into a random walk on $\Z^2$ and that satisfy appropriate moment conditions. Also (H$_{\mathrm{mart}}$) is always satisfied if $(\eta_N(i))_{i,N}$ is an i.i.d. family of centered random variables with unit variance and $\E[|\eta_1(0)|^{6+\delta}] < \infty$.

We consider the solution to
\begin{equation}\label{eq:lattice pam math}
   \partial_t v_N (t, i) = (\Delta_{\mathrm{rw}} v_N) (t, i) + \varepsilon v_N (t, i) \eta_N (i),\quad (t,i)\in[0,+\infty)\times \Z_N^2,
\end{equation}
and our aim is to show that under appropriate rescaling and renormalization it converges to a continuum limit. To even state such a convergence result, we first have to extend the rescaled solution from $\T^2_N := (\varepsilon \Z_N)^2$ to the continuous space $\T^2 := \R / (2\pi \Z)$. While a posteriori we will obtain the same limit for all ``reasonable'' functions on $\T^2$ that agree with the solution in the points of the lattice $\T_N^2$, there is one extension for which we can directly write down a closed equation and with which we will work throughout. Namely, we will use the discrete Fourier
transform~\cite{Hairer_Maas_2012,Hairer_Maas_2014,Mourrat2014}. For $\varphi \colon \T^2_N
\rightarrow \C$ we define
\[ \CF_{\T^2_N} \varphi (k) = \varepsilon^2
   \sum_{| \ell |_{\infty} < N / 2} \varphi (\varepsilon \ell) e^{- i \langle k, \varepsilon \ell \rangle}, \hspace{2em} k \in \Z^2_N, \]
where $| \ell |_{\infty}$ denotes the supremum norm on $\Z^2$. Set now
\[
   \mathcal{E}_N \varphi (x) = (2 \pi)^{- 2}
   \sum_{| k |_{\infty} < N / 2} \CF_{\T^2_N} \varphi (k) e^{i
   \langle k, x \rangle}, \hspace{2em} x \in \T^2,
\]
so that $\mathcal{E}_N \varphi$ is the function on $\T^2$ with Fourier
transform $\CF \mathcal{E}_N \varphi (k) = \CF_{\T^2_N} \varphi (k)
\1_{| k |_{\infty} < N / 2}$, $k \in \Z^2$. Then $\CE_N{\varphi} (x) = \varphi (x)$ for all $x \in
\T_N^2$ and by construction $\CE_N{\varphi}$ is infinitely smooth. If $\varphi$
is real valued, then so is $\CE_N{\varphi}$. 

\medskip

We are now able to state the hypothesis on our initial conditions:

\begin{description}
  \item[(H$_{\mathrm{init}}$)] There exists $\theta \in \R$ and $p>0$ such that the initial conditions $(v_N^0(i): i \in \Z_N^2)$ satisfy
  \[
     \sup_{N \in \N} \|  \varepsilon^{\theta} \CE_N v_N^0(\cdot / \varepsilon) \|_{L^p(\Omega; B_{1,\infty}^{0})} < \infty,
  \]
  and such that $( \varepsilon^{\theta} \CE_N v_N^0(\cdot / \varepsilon))_N$ converges in distribution in $B^{0}_{1,\infty}$ to a limit $u^0$. Here $B_{1,\infty}^{0}$ denotes a Besov space which will be defined in Section~\ref{sec:paracontrolled} below.
\end{description}

Two of the most important initial conditions for the lattice parabolic Anderson model are the constant function $v_N^0 \equiv 1$ which satisfies (H$_{\mathrm{init}}$) with $\theta = 0$, $p=\infty$ and $u^0 \equiv 1$, and the Kronecker delta $v^0_N(i) = \delta_{i,0}$, which satisfies (H$_{\mathrm{init}}$) with $\theta = -2$, $p = \infty$ and $u^0(x) = \delta(x)$, where $\delta$ denotes the Dirac delta in 0. The reason for working in the scale of spaces $B_{1,\infty}^\alpha$ rather than the more commonly used $B_{\infty, \infty}^\alpha$ is that it allows us to treat the Dirac delta, which in dimension $d$ is in $B^{-d(1-1/q)}_{q,\infty}$. It would be possible to relax the conditions on the initial condition and to allow anything with regularity better than $B^{-1+2/p}_{1,\infty}$, where $p$ is the integrability index of our potential. But since we do not see any application for this and since it would slightly complicate the notation we restrict ourselves to the case $u^0 \in B^0_{1,p}$.

\medskip

Let us rescale and renormalize $v_N$ by setting
\begin{equation}\label{eq:scaled lattice pam math}
   u_N(t,x) := e^{- t c_N} \varepsilon^\theta v_N ( t/ \varepsilon^2, x / \varepsilon), \hspace{2em} (t, x) \in \R_+ \times \T_N^2
\end{equation}
for
\begin{equation}
   c_N := (2 \pi)^{- 2} \sum_{| k |_{\infty} < N / 2} \frac{\1_{k \neq 0}}{| k |^2} \simeq \log N.
\end{equation}

\begin{lemma}   
The extension $\mathcal{E}_N u_N$ of the rescaled and renormalized process $u_N$ solves
\[
   \partial_t \mathcal{E}_N u_N = \Delta_{\mathrm{rw}}^N ( \mathcal{E}_N u_N) + \Pi_N (\mathcal{E}_N u_N \xi_N) - c_N ( \mathcal{E}_N u_N), \qquad \mathcal{E}_N u_N(0) = \varepsilon^{\theta} \CE_N v_N^0(\cdot / \varepsilon)
\]
with
\begin{gather*}
   \Delta_{\mathrm{rw}}^N \varphi (x) = \varepsilon^{-2} \int_{\Z^2} \varphi \left( x + \varepsilon y \right) \mu (\mathd y), \hspace{2em}
   \xi_N (x) = \varepsilon^{-1} (\mathcal{E}_N \eta_N(\cdot/\varepsilon)) (x), \\
   \Pi_N \varphi (x) = (2 \pi)^{- 2} \sum_{k \in \Z^2} e^{i\langle  k^N, x\rangle} \CF \varphi (k),
\end{gather*}
where
\[
   (k^N)_r = \arg \min \{ |\ell| : \ell = k_r + j N \text{ for some } j \in \Z\} \in (-N/2, N/2), \qquad r = 1,2.
\]
\end{lemma}

\begin{proof}
   Start by noting that
   \[
      \CF_{\T_N^2} (\Delta_{\mathrm{rw}}^N \varphi) (k) = \varepsilon^{-2} \int_{\Z^2} e^{i \left\langle k, \varepsilon j \right\rangle} {\mu} (\mathd j)  \CF_{\T^2_N} \varphi (k),
   \]
   so $\Delta_{\mathrm{rw}}^N$ is a Fourier multiplication operator and therefore it commutes with $\mathcal{E}_N$. This leads to
   \[
      \partial_t \mathcal{E}_N u_N (t, x) = 
   \Delta_{\mathrm{rw}}^N ( \mathcal{E}_N u_N) (t, x) + \CE_N (u_N \varepsilon^{-1} \eta_N( \cdot / \varepsilon))(t,x) - c_N ( \mathcal{E}_N u_N) (t, x).
   \]
   It remains to show that for $\varphi, \psi \colon \T_N^2 \to \C$ we have $\CE_N(\varphi \psi) = \Pi_N (\CE_N \varphi \CE_N \psi)$, which can be verified by a direct computation using 
   \begin{equation}\label{eq:product discrete fourier}
      \CF_{\T_N^2} (\varphi \psi) (k) = (2 \pi)^{- 2} \sum_{| \ell |_{\infty} < N / 2} \CF_{\T^2_N} \varphi (\ell) \CF_{\T^2_N} \psi (k - \ell);
   \end{equation}
   see also Section~8 of~\cite{Gubinelli2014KPZ}.
\end{proof}

\begin{theorem}\label{thm:main}
   Make assumptions (H$_{\mathrm{rw}}$), (H$_{\mathrm{mart}}$) and (H$_{\mathrm{init}}$) and let $T>0$.    Then 
  $\mathcal{E}_N u_N$ converges in distribution in $C([0,T], B^0_{1,\infty})$ 
  to the paracontrolled solution $u$ of the continuous equation
  \[ \LL u = (\partial_t - \Delta) u = u \diamond \xi = u\xi - u \infty, \hspace{2em} u (0) = u_0, \]
  where $\xi$ is a space white noise on $\T^2$.
  \end{theorem}
  
\begin{proof}
   In Proposition~\ref{prop:deterministic discrete limit} we show that if $(u_0^N, \xi_N, X_N \reso \xi_N - c_N, A_N)$ converges to $(u_0, \xi, X\diamond \xi, 0)$ in $\CC_1^0 \times \CC^\alpha_\infty \times \CC^{2\alpha-2}_\infty \times L(\CC_1^\alpha, \CC_1^{2\alpha-2})$, then the solution $u_N$ to
   \[ \partial_t u_N = \Delta^N_{\mathrm{rw}} u_N + \Pi_N( u_N  \xi_N) - c_N u_N, \hspace{2em} u_N (0) = u_N^0, \]
   converges to $u$. In Corollary~\ref{cor:stoch data convergence} it is shown that $(u_0^N, \xi_N, X_N \reso \xi_N - c_N)$ converges to $(u_0, \xi, X\diamond \xi)$ in distribution in $\CC_1^0 \times \CC^\alpha_\infty \times \CC^{2\alpha-2}_\infty$. Now observe that while $\CC_1^0 \times \CC^\alpha_\infty \times \CC^{2\alpha-2}_\infty$ is not separable, the support of $(u_0, \xi, X\diamond \xi, 0)$ is contained in the closure of the smooth functions in that space, and this is a Polish space. Therefore, we can apply the Skorokhod representation theorem to find a new probability space and new $(\tilde u_0^N, \tilde \xi_N, \tilde X_N \reso \tilde \xi_N - c_N)$ with the same distribution as before and which converge almost surely. It then remains to observe that in Lemma~\ref{lem:random operator bound} the convergence of $A_N$ to $0$ in probability in $L(\CC_1^\alpha, \CC_1^{2\alpha-2})$ is shown, and since $\tilde A_N$ has the same distribution as $A_N$ it must also converge to $0$ in probability. This concludes the proof.
\end{proof}

\section{Paracontrolled analysis of the discrete equations}\label{sec:paracontrolled}

Abusing notation, we denote the extension $\CE_N u_N$ from now on simply by $u_N$, and we take the equation
\begin{equation}
  \label{eq:discrete pam} 
 \partial_t u_N = 
   \Delta_{\mathrm{rw}}^N u_N + \Pi_N (u_N \xi_N) - c_N u_N, \quad  u_N(0,x)= u_N^0(x)
\end{equation}
with $u_N^0 = \varepsilon^{\theta} \CE_N v_N^0(\cdot / \varepsilon)$ as the starting point of our analysis. We shall use the paracontrolled analysis developed in~{\cite{Gubinelli2013}} to derive a priori bounds on the solution which depend on norms of $\xi_N$ and $u_0^N$ that stay uniformly bounded in $N$. This will allow us to deduce the convergence. Let us start by briefly recalling the basics of paracontrolled distributions.

\subsection{Paracontrolled distributions and the continuous PAM}\label{sec:paracontrolled continuous}

Here we recall the basics of paracontrolled distributions, for an introduction see also the lecture notes~{\cite{Gubinelli2014EBP}}, and we sketch how to solve the continuous parabolic Anderson model in dimension 2.

\medskip

Throughout, we fix a Littlewood-Paley decomposition $(\Delta_j)_{j
\ge - 1}$, where
\[ \Delta_j u = \rho_j (\mathD) u = \CF^{- 1} \left( \rho_j \CF u \right) \]
with $\rho_j = \chi$ if $j = - 1$ and $\rho_j = \rho (2^{- j} \cdot)$ if $j
\ge 0$, for nonnegative radial functions $\chi, \rho \in C^{\infty}
(\R^d, \R)$, where $\rho$ is supported in a ball $\CB = \{ |
x | \le c \}$ and $\rho$ is supported in an annulus $\CA = \{ a
\le | x | \le b \}$ for suitable $a, b, c > 0$, such that
\begin{enumerate}
  \item $\chi + \sum_{j \ge 0} \rho (2^{- j} \cdot) \equiv 1$ and
  
  \item \label{property-2-dyadic}$\tmop{supp} (\chi) \cap \tmop{supp} (\rho
  (2^{- j} \cdot)) \equiv 0$ for $j \ge 1$ and $\tmop{supp} (\rho
  (2^{- i} \cdot)) \cap \tmop{supp} (\rho (2^{- j} \nosymbol \cdot))
  \equiv 0$ for all $i, j \ge 0$ with $| i - j | \ge 1$.
\end{enumerate}
We also use the notation
\[ \Delta_{\le j} f = \sum_{i \le j} \Delta_i f \]
as well as $K_i = \CF^{- 1} \rho_i$ so that
\[ K_i \ast f = \CF^{- 1} \left( \rho_j \CF f \right) = \Delta_i f. \]
For $\alpha \in \R$, the space $\CC^{\alpha}_p$ is
defined as $\CC^{\alpha}_p = B^{\alpha}_{p, \infty}$, where
\[ B^{\alpha}_{p, q} = B^{\alpha}_{p, q} (\T^d) = \left\{ f \in \CS'
   (\T^d) : \|f\|_{B^{\alpha}_{p, q}} = \| (2^{j \alpha} \| \Delta_j
   f\|_{L^{p}})_j \|_{\ell^q} < \infty \right\}, \]
and we write $\| \cdot \|_{\CC^{\alpha}_p} = \|
\cdot \|_{B^{\alpha}_{p, \infty}}$. We will need the following
embedding theorem for Besov spaces:

\begin{lemma}
  \label{lem:besov embedding}(Besov embedding) Let $1 \le p_1 \le
  p_2 \le \infty$ and $1 \le q_1 \le q_2 \le \infty$,
  and let $\alpha \in \R$. Then $B^{\alpha}_{p_1, q_1}$ is
  continuously embedded into $B^{\alpha - d (1 / p_1 - 1 / p_2)}_{p_2, q_2}$.
\end{lemma}

The product of two distributions can be (at least formally) decomposed as
\[ fg = \sum_{j \ge - 1} \sum_{i \ge - 1} \Delta_i f \Delta_j g =
   f \para g + f \lpara g + f \reso g. \]
Here $f \para g$ is the part of the double sum with $i < j - 1$, $f \lpara g$
is the part with $i > j + 1$, and $f \reso g$ is the ``diagonal'' part, where
$|i - j| \le 1$. More precisely,
\[ f \para g = g \lpara f = \sum_{j \ge - 1} \sum_{i = - 1}^{j - 2}
   \Delta_i f \Delta_j g = \sum_{j \ge -1} \Delta_{\le j-2} f \Delta_j g \hspace{2em} \text{and} \hspace{2em} f \reso g =
   \sum_{|i - j| \le 1} \Delta_i f \Delta_j g. \]
We call $f \para g$ and $f \lpara g$ \tmtextit{paraproducts}, and $f \reso g$
the \tmtextit{resonant} term. Bony's~{\cite{Bony1981}} observed that $f \para g$ (and thus $f
\lpara g$) is always a well-defined distribution and the only difficulty in
constructing $fg$ for arbitrary distributions lies in handling the diagonal
term $f \reso g$.

\begin{theorem}[Bony's paraproduct estimates, \cite{Gubinelli2014KPZ}, Lemma~6.1]\label{thm:paraproduct}
  Let $p \in [1,\infty]$, 
  $\beta \in \R$ and $f, g \in \CS'$. Then
  \begin{equation}
    \label{eq:para-1} \|f \para g\|_{\CC^\beta_p} \lesssim \min\{\|f\|_{L^{p}} \|g\|_{\CC^\beta_\infty}, \|f\|_{L^{\infty}} \|g\|_{\CC^\beta_p}\},
  \end{equation}
  and for $\alpha < 0$ furthermore
  \begin{equation}
    \label{eq:para-2} \|f \para g\|_{\CC^{\alpha + \beta}_p} \lesssim \min\{ \|f\|_{\CC^\alpha_\infty}  \|g\|_{\CC^\beta_p} , \|f\|_{\CC^\alpha_p}  \|g\|_{\CC^\beta_\infty}\}.
  \end{equation}
  If $\alpha + \beta > 0$, we also have
  \begin{equation}
    \label{eq:para-3} \|f \reso g\|_{\CC^{\alpha + \beta}_p} \lesssim \min\{ \|f\|_{\CC^\alpha_p}  \|g\|_{\CC^\beta_\infty}, \|f\|_{\CC^\alpha_\infty} \|g\|_{\CC^\beta_p} \} .
  \end{equation}
\end{theorem}

\begin{corollary}\label{cor:product}
  Let $p \in [1,\infty]$ and $f \in \CC^{\alpha}_p$ and $g \in \CC^{\beta}_\infty$ with
  $\alpha + \beta > 0$. Then the product $(f, g) \mapsto fg$ is a bounded
  bilinear map from $\CC^{\alpha}_p \times \CC^{\beta}_\infty$ to $\CC^{\alpha \wedge
  \beta}_p$.
\end{corollary}

The main idea of~{\cite{Gubinelli2013}} is that the paraproduct $f \para g$ is
a ``frequency modulation'' of $g$, and thus on small scales resembles $g$. By the philosophy of controlled paths~\cite{Gubinelli2004} we should  be able to control $(f \para g) h$ for some given $h$ provided that we have an a priori control on $g h$. Making these heuristics rigorous is the main achievement of the theory of paracontrolled distributions, and doing so is possible with the help of the following commutator estimate which is a generalization of one of the main results in~\cite{Gubinelli2013}.

\begin{lemma}[\cite{ProemelTrabs2015}, Lemma~4.4]
   Define the commutator $C(f,g,h) = (f \para g) \reso h - f (g \reso h)$. Then we have for all $p \in [1,\infty]$ and $\alpha <1$, $\beta, \gamma \in \R$ with $\beta + \gamma < 0 < \alpha + \beta + \gamma$ the bound
   \[
      \| C(f,g,h) \|_{\CC^{\alpha}_p} \lesssim \| f \|_{\CC^\alpha_p} \| g \|_{\CC^\beta_\infty} \| h \|_{\CC^\gamma_\infty}.
   \]
\end{lemma}

Let us define for $p \in [1,\infty]$ and $\gamma \ge 0$ the space $\mathcal{M}^{\gamma}_T L^p = \{ v \colon [0, T] \rightarrow \CS'(\T^2) : \| v \|_{\mathcal{M}^{\gamma}_T L^p} < \infty \}$, where
\[
   \| v \|_{\mathcal{M}^{\gamma}_T L^p} = \sup_{t \in [0, T]} \{ \| t^{\gamma} v(t) \|_{L^p} \} .
\]
If further $\alpha \in (0,
2)$ and $T > 0$ we define the norm
\[
   \| f \|_{\LL^{\gamma, \alpha}_p(T)} = \max \big\{ \| t \mapsto t^{\gamma} f(t) \|_{C^{\alpha / 2}_T L^{p}}, \| f \|_{\mathcal{M}^{\gamma}_T \CC_p^{\alpha}} \big\}
\]
and the space $\LL^{\gamma, \alpha}_p(T) = \{ f \colon [0, T] \rightarrow
\CS': \| f \|_{\LL^{\gamma, \alpha}_p(T)} < \infty \}$ as well as
\[
   \LL^{\gamma,\alpha}_p = \big\{f\colon \R_+\to \CS': f|_{[0,T]} \in \LL^{\gamma,\alpha}_p(T) \text{ for all } T > 0\big\}.
\]
It will be convenient to introduce a modified paraproduct. Let $\varphi \in C^{\infty} (\R, \R_+)$ be nonnegative with
compact support contained in $\R_+$ and with total mass $1$, and
define for all $i \ge - 1$ the operator
\[ Q_i : C \CC^{\beta} \rightarrow C \CC^{\beta}, \hspace{2em} Q_i f (t) =
   \int_{0}^\infty 2^{- 2 i} \varphi (2^{2 i} (t - s)) f (s) \mathd
   s. \]
We will often apply $Q_i$ and other operators on $C \CC^{\beta}$ to functions
$f \in C_T \CC^{\beta}$ which we then simply extend from $[0, T]$ to
$\R_+$ by considering $f (\cdot \wedge T)$. With the help of $Q_i$,
we define the modified paraproduct
\[ f \mpara g = \sum_i (Q_i \Delta_{\le i - 2} f) \Delta_i g \]
for $f, g \in C \left( \R_+, \CS' \right)$. If $f$ or $g$ has a blow-up at zero which is integrable, we still define $f \mpara g$ in the same way.

\begin{lemma}[\cite{Gubinelli2014KPZ}, Lemmas~6.4, 6.5, 6.7]\label{lem:modified paraproduct exp}
  For any $\beta \in \mathbb{R}$, $p \in [1,\infty]$, $\gamma \in [0, 1)$, and $t>0$ we have
  \begin{equation}\label{eq:mod para-1 exp}
    t^{\gamma} \|f \mpara g (t) \|_{\CC_p^\beta} \lesssim \|f\|_{\mathcal{M}^{\gamma}_t L^{p}} \|g (t) \|_{\CC^\beta_\infty},
  \end{equation}
  and for $\alpha \in (0,2)$ furthermore
  \[
      t^{\gamma} \| (f \mpara g - f \para g) (t) \|_{\CC_p^{\alpha + \beta}} \lesssim \| f \|_{\LL^{\gamma, \alpha}_p(t)} \| g (t) \|_{\CC^\beta_\infty},
   \]
  as well as
  \[
     t^{\gamma} \left\| ( \LL (f \mpara g) - f \mpara ( \LL g ) ) (t) \right\|_{\CC_p^{\alpha + \beta - 2}} \lesssim \| f   \|_{\LL^{\gamma, \alpha}_p(t)} \| g (t) \|_{\CC^\beta_\infty},
  \]
  and for $\delta>0$ also
  \[
     \| f \mpara g \|_{\LL^{\gamma,\alpha}_p(T)} \lesssim \| f \|_{\LL^{\gamma,\delta}_p(T)} ( \| g \|_{C_T \CC^{\alpha}_\infty} + \left\| \LL g \right\|_{C_T \CC^{\alpha - 2}_\infty} ).
  \]
\end{lemma}

Finally we need the Schauder estimates for the Laplacian. We write $I f (t) = \int_0^t P_{t - s} f (s) \mathd s$.

\begin{lemma}[Schauder estimates, \cite{Gubinelli2014KPZ}, Lemma~6.6]\label{lemma:schauder exp}
   Let $\alpha \in (0, 2)$, $p \in [1,\infty]$, and $\gamma \in [0, 1)$. Then
  \begin{equation}\label{eq:schauder-heat exp}
     \|I f\|_{\LL^{\gamma, \alpha}_p(T)} \lesssim \| f \|_{\mathcal{M}^{\gamma}_T \CC^{\alpha - 2}_p}
  \end{equation}
  for all $T > 0$. If further $\beta \ge - \alpha$, then
  \begin{equation}\label{eq:schauder initial contribution exp}
     \| s \mapsto P_s u_0 \|_{\LL^{(\beta + \alpha) / 2, \alpha}_p(T)} \lesssim \| u_0 \|_{\CC_p^{- \beta}} .
  \end{equation}
  For all $\alpha \in \R$, $\gamma \in [0, 1)$, and $T > 0$ we have
  \begin{equation}\label{eq:schauder without time hoelder exp}
     \|I f\|_{\mathcal{M}^{\gamma}_T \CC^{\alpha}_p} \lesssim \| f \|_{\mathcal{M}^{\gamma}_T \CC^{\alpha - 2}_p}.
  \end{equation}
\end{lemma}

For the remainder of this subsection we fix $\alpha \in (2/3,1)$.

\begin{definition}
  Let  $X \in \CC^\alpha_\infty$. We define the space $\DD^{\alpha}_{X}$ of distributions paracontrolled by $X$ as the set of
  all $(u,u^X,u^\sharp) \in C \CC^{0}_1 \times \LL^{\alpha/2,\alpha}_1 \times \LL^{{\alpha},2\alpha}_1$ such that
  \[
     u = u^X \mpara X + u^\sharp.
  \]
  For $T > 0$ we set $\CD^{\alpha}_{X} (T) = \CD^{\alpha}_{X} |_{[0, T]}$, and
  we define
  \[ \| u \|_{\DD^{\alpha}_{X} (T)} = \| u^X \|_{\LL^{\alpha/2,\alpha}_1(T)} + \| u^{\sharp} \|_{\LL_1^{\alpha,2 \alpha}(T)} . \]
  If $\tilde{X} \in \CC^\alpha_\infty$ and $(\tilde u,\tilde{u}^{\tilde X},\tilde{u}^\sharp) \in \DD^\alpha_{\tilde X}$, we write
  \[
     d_{\DD^\alpha(T)}(u, \tilde u) = \| u^X - \tilde{u}^{\tilde X} \|_{\LL^{\alpha/2,\alpha}_1(T)} + \| u^{\sharp} - \tilde{u}^\sharp \|_{\LL_1^{\alpha,2 \alpha}(T)}.
  \]
  Abusing notation, we will sometimes write $u \in \DD^\alpha_X$ rather than $(u,u^X,u^\sharp) \in \DD^\alpha_X$.
\end{definition}

For $(u,u^X, u^\sharp) \in \DD^\alpha_X$ we expand
\[
   u \xi = u \para \xi + u \lpara \xi + u^\sharp \reso \xi + (u^X \mpara X - u^X \para X) \reso \xi + C(u^X, X, \xi) + u^X (X \reso \xi),
\]
and given $\xi \in \CC^{\alpha-2}_\infty$, the right hand side is under control provided that we can bound $X \reso \xi$ in $\CC^{2\alpha-2}_\infty$. Moreover, in that case we have
\[
   u \xi - u \para \xi \in \mathcal{M}^{\alpha}\CC^{2\alpha-2}_1.
\]
If now $v$ denotes the solution to $\LL v = u \xi$, $v(0) = u^0$, then we make the paracontrolled ansatz $v = u \mpara X + v^\sharp$ and obtain
\[
   \LL v^\sharp = \LL v - \LL (u \mpara X) = u \xi - [\LL (u \mpara X) - u \mpara \LL X] + [u \mpara \LL X - u \para \xi].
\]
So if $\LL X - \xi \in \CC^{2\alpha-2}_\infty$ (and we will always take $X = \Delta^{-1}(\xi - (2\pi)^{-2} \CF \xi(0))$ for which $\LL X - \xi  = (2\pi)^{-2} \CF \xi(0) \in C^\infty$), then we can control the right hand side in $\mathcal{M}^{\alpha}\CC^{2\alpha-2}_1$, and since $v^\sharp(0) = u^0 - u(0) \mpara X \in \CC^0_1$, we get from the Schauder estimates that $v^\sharp \in \LL_1^{\alpha,2\alpha}$. This allows us to set up a Picard iteration in $\DD^\alpha_X(T)$ for a sufficiently small $T>0$ and to obtain a unique solution $u$ to our equation. Since the equation is linear, the length $T$ of the time interval does not depend on the initial condition, and iterating this construction we obtain a unique solution $u \in \DD^\alpha_X$ which is defined on all of $\R_+$ -- always under the assumption that $X \reso \xi \in \CC^{2\alpha-2}_\infty$ is given. In that case the solution also depends continuously on the data $(\xi, X, X\reso \xi, u^0)$, because all the operations on the right hand side of the equation are continuous.

\medskip

But note that in our setting we have $2\alpha-2 < 0$, which means that $X \reso \xi$ cannot be controlled using Bony's estimates (or other analytic tools), and we have to include it as an additional part of the data of the problem. Moreover, so far our entire analysis was pathwise and dimension independent, but now we want to use that $\xi$ is a space white noise in dimension 2 in order to use probabilistic estimates to bound $X \reso \xi$. And as it turns out is is not possible to directly make sense of this term. Rather we have to perform a Wick renormalization and consider
\[
   X \diamond \xi = X\reso \xi - \infty = \lim_{\delta \to 0} (\rho_\delta \ast X) \reso (\rho_\delta \ast \xi) - c_\delta,
\]
where $\rho_\delta = \delta^{-2} \rho(\delta^{-1} \cdot)$, $\rho$ is a mollifier, and $(c_\delta)$ a family of diverging constants such that
\[
   c_\delta = \frac{1}{2\pi} \log(\frac{1}{\delta}) + O(1)
\]
and only the finite contribution $O(1)$ depends on the specific mollifier $\rho$. Thus, we obtain the following result.
\begin{proposition}[see also Corollary~5.9 in~\cite{Gubinelli2013}]\label{prop:continuous pam}
   Let $\alpha \in (2/3,1)$ and let $(\xi, X, X\diamond \xi) \in \CC^{\alpha-2}_\infty \times \CC^{\alpha}_\infty \times \CC^{2\alpha-2}_\infty$ be such that $- \Delta X = \xi - (2\pi)^{-2} \CF \xi(0)$, and let $u^0 \in \CC^0_1$. Then there exists a unique solution $u \in \DD^\alpha_X$ to the equation
   \[
      \LL u = u\diamond \xi :=  u \para \xi + u \lpara \xi + u^\sharp \reso \xi + (u \mpara X - u \para X) \reso \xi + C(u, X, \xi) + u (X \diamond \xi), \qquad u(0) = u^0.
   \]
   Moreover, $u$ depends continuously on $(\xi, X, X \diamond \xi, u^0)$. If $X \diamond \xi = \lim_{\delta \to 0} (\rho_\delta \ast X) \reso (\rho_\delta \ast \xi) - c_\delta$, then $u = \lim_{\delta \to 0} u_\delta$, where
   \[
      \LL u_\delta = u_\delta (\rho_\delta \ast \xi) - u_\delta c_\delta,\qquad u_\delta(0) = u^0.
   \]
   If $d=2$ and $\xi$ is a space white noise, then almost surely all of the above conditions are satisfied, $X \diamond \xi$ can be chosen independently of the mollifier $\rho$, and we have
   \[
      c_\delta = (2\pi)^{-2} \sum_{k \in \Z^2 \setminus\{0\}} \frac{|\CF \rho(\delta k)|^2}{|k|^2} \simeq |\log \delta|.
   \]
\end{proposition}

\subsection{Estimation of the discrete operators}

To extend the previous discussion to the lattice equation we will need to derive bounds on the discrete Laplacian and its semigroup, and also on the operator $\Pi_N$. Let us point out that all the estimates presented in this section have already been established in~\cite{Gubinelli2014KPZ}, Chapter~8, in the one dimensional setting and the extension to higher dimensions follows from the same arguments with only notational modifications which is why we omit most of the proofs. Throughout this subsection we fix $d=2$.

\paragraph{Estimates for the discrete Laplacian}

Recall that $\Delta_{\mathrm{rw}}^N \varphi(x) = \varepsilon^{-2} \int_{\Z^2} \varphi(x + \varepsilon j) \mu(\dd j)$ and let us write
\[ f (x) = \frac{\int_{\Z^2} e^{i \langle x, j \rangle} {\mu}
   (\mathd j)}{- | x |^2}, \]
so that
\[ \CF \Delta_{\mathrm{rw}}^N \varphi (k) = - | k |^2 f (k \varepsilon) \CF \varphi (k). \]

\begin{lemma}
   Under the hypothesis  (H$_{\mathrm{rw}}$) there exist a constant $c_f>0$ with $f(x)\geq c_f$ for all $x\in[-\pi,\pi]^2$. 
\end{lemma}
\begin{proof}
   Since the measure $\mu$ is radial we have
   \[
      \int_{\Z^2} e^{i \langle x, j \rangle} {\mu}(\mathd j) = \frac{1}{2} \int_{\Z^2} (e^{i \langle x, j \rangle} + e^{-i\langle x, j \rangle}) {\mu}(\mathd j) =  \int_{\Z^2} \cos( \langle x, j \rangle) {\mu}(\mathd j),
   \]
   and using that $\mu$ has total mass zero we get
   \[
      f(x) = \int_{\Z^2 \setminus \{0\}} \frac{1 - \cos(\langle x, j \rangle)}{|x|^2} \mu(\dd j).
   \]
   Now $\mu$ restricted to $\Z^2 \setminus \{0\}$ is a positive measure and the integrand is nonnegative. Moreover, $\mu$ is radial and therefore $\mu(\{(1,0)\}) = \mu(\{(0,1)\}) > 0$, which leads to
   \[
      f(x) \ge \frac{2 - \cos(x_1) - \cos(x_2)}{|x|^2} \mu(\{(0,1)\}) = \frac{\sin^2(x_1/2) + \sin^2(x_2/2)}{|x_1|^2+|x_2|^2} \mu(\{(0,1)\}).
   \]
   Now it suffices to note that for every $a\in (0,\pi)$ there exists $b>0$ with $|\sin(x)| \ge b |x|$ for all $x \in [-a,a]$.
\end{proof}

\begin{lemma}[\cite{Gubinelli2014KPZ}, Lemma~8.4]
  Let $\mu$ satisfy $(H_{\mathrm{rw}})$. Then the function
  \[ f (x) = - \frac{\int_{\R^2} e^{i \langle x, y \rangle} {\mu}
     (\mathd y)}{| x |^2} = \int_{\R^2} \frac{1 - \cos (\langle x, y
     \rangle)}{| \langle x, y \rangle |^2} \frac{| \langle x, y \rangle |^2}{|
     x |^2 | y |^2} | y |^2 {\mu} (\mathd y) \]
  is in $C^4_b$ and such that $f (0) = 1$.
\end{lemma}

\begin{lemma}[\cite{Gubinelli2014KPZ}, Lemma~8.10]\label{lem:discrete laplacian}
  Let ${\mu}$ satisfy $(H_{\mathrm{rw}})$, $\alpha <
  1$, $\beta \in \R$, $p \in [1,\infty]$ and let $\varphi \in \CC_p^{\alpha}$ and $\psi\in \CC^{\beta}_\infty$. Then for all $\delta \in [0, 1]$ and $N \in \N$
  \[ \| \Delta_{\mathrm{rw}}^N \varphi - \Delta \varphi \|_{\mathscr C^{\beta - 2 - \delta}_p} \lesssim N^{-
     \delta} \| \psi \|_{\mathscr C_p^\beta}. \]
\end{lemma}

While in general the semigroup generated by the discrete Laplacian $\Delta_{\mathrm{rw}}^N$
does not have good regularizing properties, we will only apply it to functions
with spectral support contained in $(- N/2, N/2)^2$ where it has the
same smoothing effect as the heat flow. It is here where we will use that  
$f (x) \geqslant c_f > 0$ for $| x |_{\infty} \leqslant \pi $.

\begin{lemma}[\cite{Gubinelli2014KPZ}, Lemma~8.11]
  Assume that ${\mu}$ satisfies$(H_{\mathrm{rw}})$. Let $\alpha \in \R$,
  $\beta \geqslant 0$, $p \in [1, \infty]$, and let $\varphi \in \CS'$ with
  $\tmop{supp} \left( \CF \varphi \right) \subset (- N/2, N/2)^2$.
  Then we have for all $T > 0$ uniformly in $t \in (0, T]$
  \begin{equation}
    \label{eq:discrete heat flow} \| e^{t \Delta_{\mathrm{rw}}^N} \varphi \|_{\mathscr C_p^{\alpha +
    \beta}} \lesssim t^{- \beta / 2} \| \varphi \|_{\mathscr C^{\alpha}_p} .
  \end{equation}
\end{lemma}

An interpolation argument allows to extend~(\ref{eq:discrete heat flow}) to
$L^{p}$, so that $\| e^{t \Delta_{\mathrm{rw}}^N} \varphi \|_{L^{p}} \lesssim t^{-
\alpha / 2} \| \varphi \|_{- \alpha}$ for all $\alpha > 0$ and all $\varphi$
with spectral support in $(- N/2, N/2)^2$.

\begin{corollary}[\cite{Gubinelli2014KPZ}, Lemma~8.12]
  Let ${\mu}$ satisfy $(H_{\mathrm{rw}})$. Let $\alpha \in (0, 2)$ and $\varphi \in
  \CC_p^{\alpha}$ with spectral support in $(- N/2, N/2)^2$. Then
  \[ \| (e^{t \Delta_{\mathrm{rw}}^N} - \tmop{id}) \varphi \|_{L^{p}} \lesssim
     t^{\alpha / 2} \| \varphi \|_{\mathscr C_p^\alpha} . \]
\end{corollary}

Combining these estimates, we can apply the same arguments as in the
continuous setting to derive analogous Schauder estimates for $(e^{t
\Delta_{\mathrm{rw}}^N})$ as in Lemma~\ref{lemma:schauder exp} -- of course always restricted to elements of $\CS'$ that are
spectrally supported in $(- N/2, N/2)^2$.

\paragraph{Fourier shuffle operator}

Let us introduce the operator $\PC_N u = \CF^{-1}(\1_{(-N/2,N/2)^2} \CF u)$, for which we have the following estimate.

\begin{lemma}[\cite{Gubinelli2014KPZ}, Lemma~8.7]\label{lem:cutoff bound}
  Let $\alpha \geqslant 0$, $p \in [1,\infty]$ and $\varphi \in \CC^{\alpha}_p$. Then for any $\delta
  \geqslant 0$
  \[
     \| \PC_N \varphi - \varphi \|_{\CC^{\alpha - \delta}_p} \lesssim N^{- \delta} (\log N)^2 \| \varphi  \|_{\CC^\alpha_p} .
  \]
\end{lemma}

As a consequence we can bound the operator $\Pi_N$:

\begin{lemma}[\cite{Gubinelli2014KPZ}, Lemma~8.8]\label{lem:periodic cutoff bound}
  Let $\alpha \geqslant 0$, $p \in [1,\infty]$ and $\varphi \in \CC^{\alpha}_p$. Then for any $\delta
  \geqslant 0$
  \[ \| \Pi_N \varphi - \varphi \|_{\CC^{\alpha - \delta}_p} \lesssim N^{- \delta} (\log N)^2 \| \varphi
     \|_{\CC^\alpha_p} . \]
  If $\tmop{supp} ( \CF \varphi ) \subset [- c N, c N]^2$ for some $c \in
  (0, 1)$, then this inequality extends to general $\alpha \in \R$.
\end{lemma}

\begin{remark}\label{rmk:PiN on product}
  There exists $c \in (0, 1)$, independent of $N$, such that if $\tmop{supp}
  \left( \CF \psi \right) \subset (- N / 2, N / 2)^2$, then $\tmop{supp} \left( \CF
  (\varphi \para \psi) \right) \subset [- c N, c N]^2$. This means that we can always
  bound $\Pi_N (\varphi \para \psi) - \varphi \para \psi$, even if the paraproduct has negative
  regularity. On the other side the best statement we can make about the
  resonant product is that if $\varphi$ and $\psi$ are both spectrally supported in $(-
  N / 2, N / 2)^2$, then $\tmop{supp} \left( \CF (\varphi \reso \psi) \right) \subset (-N, N)^2$. A simple consequence is that if $\alpha + \beta > 0$, $\varphi \in
  \CC^{\alpha}_p$, $\psi \in \CC^{\beta}_\infty$, and $\mathrm{supp}( \CF \varphi) \cup \mathrm{supp}( \CF \psi) \subset (- N / 2, N / 2)^2$, then
  \[
     \| \Pi_N (\varphi\psi) - \varphi\psi \|_{\CC_p^{\alpha \wedge \beta - \delta}} \lesssim N^{- \delta} (\log N)^2 \| \varphi \|_{\CC^\alpha_p} \| \psi \|_{\CC^\beta_\infty} .
  \]
\end{remark}

Finally we need to commute $\Delta_{\mathrm{rw}}^N$ with $\Pi_N$, which in general is not possible but in our setting can be done by relying on the discrete structure that is implicit in the background.

\begin{lemma}\label{lem:shuffle laplacian commutator}
  Let $\alpha < 1$, $\beta \in \R$, $p \in [1,\infty]$ and let $\varphi \in \CC^{\alpha}_p$, $\psi
  \in \CC^{\beta}_\infty$ have spectral support in $(- N / 2, N / 2)^2$. Then for all $\delta>0$
  \[ \| \Delta_{\mathrm{rw}}^N \Pi_N  (\varphi \para \psi) - \Pi_N (\varphi \para \Delta_{\mathrm{rw}}^N \psi)
     \|_{\CC^{\alpha + \beta - 2 - \delta}_p} \lesssim \| \varphi \|_{\CC^\alpha_p} \| \psi
     \|_{\CC^\beta_\infty}. \]
\end{lemma}

\begin{proof}
  If $g$ and $h$ have spectral support in $(- N / 2, N / 2)^2$, there are two
  unique lattice functions $\tilde{g}$ and $\tilde{h}$ such that $g
  =\mathcal{E}_N \tilde{g}$ and $h =\mathcal{E}_N \tilde{h}$ and therefore
  \[
     \Delta_{\mathrm{rw}}^N \Pi_N  (g h) = \Delta_{\mathrm{rw}}^N \Pi_N  (\mathcal{E}_N \tilde{g}
     \mathcal{E}_N \tilde{h}) = \Delta_{\mathrm{rw}}^N \mathcal{E}_N  (\tilde{g}
     \tilde{h}) =\mathcal{E}_N \Delta_{\mathrm{rw}}^N  (\tilde{g} \tilde{h}),
  \]
  and on the other side a direct computation shows that
  \[
    \Delta_{\mathrm{rw}}^N  (\tilde{g} \tilde{h}) =  (\Delta_{\mathrm{rw}}^N  \tilde{g}) \tilde{h}
    +  \tilde{g} \Delta_{\mathrm{rw}}^N  \tilde{h} + \varepsilon^{- 2} \int 
    (\tilde{g}(\cdot + \varepsilon j) - \tilde{g})(\tilde{h}(\cdot + \varepsilon j) - \tilde{h}) \mu (\mathd j).
  \]
  We apply this with $g = \Delta_{\le k-2} \varphi$ and $h = \Delta_k \psi$ and sum over $k$ to obtain
  \[
     \Delta_{\mathrm{rw}}^N \Pi_N  (\varphi \para \psi) = \Pi_N ((\Delta_{\mathrm{rw}}^N \varphi) \para \psi)
    + \Pi_N(\varphi \para \Delta_{\mathrm{rw}}^N  \psi) + \varepsilon^{- 2} \int 
    \Pi_N[(\varphi(\cdot + \varepsilon j) - \varphi) \para(\psi(\cdot + \varepsilon j) - \psi) ] \mu (\mathd j).
  \]
  Combining Lemma~\ref{lem:discrete laplacian} and Remark~\ref{rmk:PiN on product} we have
  \[
     \| \Pi_N ((\Delta_{\mathrm{rw}}^N \varphi) \para \psi)\|_{\CC^{\alpha + \beta - 2 - \delta}_p} \lesssim \| \Delta_{\mathrm{rw}}^N \varphi \|_{\CC^{\alpha - 2}_p} \| \psi \|_{\CC^\beta_\infty} \lesssim \| \varphi \|_{\CC^{\alpha}_p} \| \psi \|_{\CC^\beta_\infty},
  \]
  while the integral can be bounded by
  \begin{align*}
     &\Big\| \varepsilon^{- 2} \int \Pi_N[(\varphi(\cdot + \varepsilon j) - \varphi) \para(\psi(\cdot + \varepsilon j) - \psi) ] \mu (\mathd j) \Big\|_{\CC^{\alpha + \beta - 2 - \delta}_p} \\
     &\hspace{60pt} \lesssim \varepsilon^{- 2} \int \| \varphi(\cdot + \varepsilon j) - \varphi \|_{\CC^{\alpha-1}_p} \| \psi(\cdot + \varepsilon j) - \psi \|_{\CC^{\beta-1}_\infty} |\mu| (\mathd j) \\
     &\hspace{60pt} \lesssim \int |j|^2 |\mu| (\mathd j) \| \varphi \|_{\CC^\alpha_p} \|\psi \|_{\CC^\beta_\infty} \lesssim \| \varphi \|_{\CC^\alpha_p} \|\psi \|_{\CC^\beta_\infty}.
  \end{align*}
  This concludes the proof.
\end{proof}

\subsection{Paracontrolled ansatz}

Let now $u_N \in C (\R_+, C^{\infty} (\T^2))$ solve
\[ \LL_N u_N = \Pi_N( u_N  \xi_N) - c_N u_N, \hspace{2em} u_N (0) = u_N^0, \]
where we wrote
\[
   \LL_N = \partial_t - \Delta_{\mathrm{rw}}^N.
\]
Here $\xi_N$ and $u_N^0$ are deterministic and fixed, and we assume that they both have spectral support in $(-N/2,N/2)^2$. To lighten the notation, in this subsection we shall omit the subscript $N$
when no confusion arises, writing for example $u, \xi, u_0$ instead of
$u_N, \xi_N, u^0_N$. Note that existence and uniqueness of $u_N$ pose no problem, because we are only working with finitely many Fourier modes and therefore our PDE is actually a linear ODE.

Let us start by making the following ansatz for $u$:
\begin{equation}\label{eq:discrete paracontrolled ansatz}
   u = \Pi_N (u^X \mpara X) + u^{\sharp},
\end{equation}
where $(u,u^X,u^\sharp) \in C \CC^{0}_1 \times \LL^{\alpha/2,\alpha}_1 \times \LL^{{(\alpha+\beta)/2},\alpha+\beta}_1$  for some $\alpha \in (2 / 3, 1-2/p)$ and $\beta \in (2 - 2 \alpha, \alpha)$, and
\[ X = \int_0^{\infty} P^N_t  (\xi - (2\pi)^{-2} \CF \xi(0)) \mathd t \]
with $(P^N_t)_{t \geqslant 0}$ denoting the heat flow generated by $\Delta_{\mathrm{rw}}^N$. Using this ansatz, we get
\begin{align*}
   \LL_N u & = \LL_N  \Pi_N (u^X \mpara X) + \LL_N u^{\sharp} = \Pi_N (u \xi) - c_N u  \\
   & = \Pi_N (u \para \xi) + \Pi_N (u \lpara \xi) + \Pi_N (u \reso \xi) - c_N u ,
\end{align*}
and therefore
\begin{align}\label{eq:LN usharp} \nonumber
   \LL_N u^{\sharp} & = \Pi_N \left\{(u \para \xi) -  (u^X \mpara \LL_N X) + (u \lpara \xi) + (u \reso \xi) - c_N u \right\} \\
   &\quad + \{ \Pi_N (u^X \mpara \LL_N X) - \LL_N  \Pi_N (u^X \mpara  X)\},
\end{align}
where we used that $\Pi_N u = u$ because $u$ has spectral support in $(-N/2, N/2)^2$. Now recall that $\beta < \alpha$ and therefore Lemmas~\ref{lem:shuffle laplacian commutator},~\ref{lem:modified paraproduct exp} and~\ref{lem:periodic cutoff bound} show that
\[
   \| \Pi_N (u^X \mpara \LL_N X) - \LL_N  \Pi_N (u^X \mpara  X) \|_{\mathcal M^{\alpha/2}_T \CC_1^{\alpha+\beta-2}} \lesssim \| u^X \|_{\LL^{\alpha/2,\alpha}_p(T)} \| X \|_{\CC^\alpha_\infty}.
\]
Moreover, $\LL_N X = \xi - (2\pi)^{-2} \CF \xi(0)$ and setting $u^X = u$ we have by Lemma~\ref{lem:modified paraproduct exp} and Remark~\ref{rmk:PiN on product}
\[
   \| \Pi_N( (u \para \xi) - (u^X \mpara \LL_N X) + u \lpara \xi) \|_{\mathcal M^{\alpha/2}_T \CC_1^{\alpha + \beta -2}} \lesssim \| u \|_{\mathcal{M}^{\alpha/2}_T \CC^\alpha_1 } \| \xi \|_{\CC^{\alpha-2}_\infty}.
\]
Now we plug in the paracontrolled ansatz for $u$ and obtain $u \reso \xi = (\Pi_N (u \mpara X)) \reso \xi + u^{\sharp} \reso \xi$, and by Lemma~\ref{lem:periodic cutoff bound}
\[
   \| \Pi_N (u^\sharp \reso \xi) \|_{\mathcal M^{(\alpha+\beta-\delta)/2}_T \CC^{2\alpha+\beta -2 \delta}_1} \lesssim \| u^\sharp \|_{\mathcal M_T^{(\alpha+\beta-\delta)/2} \CC^{\alpha+\beta-\delta}_1} \|\xi\|_{\CC^{\alpha-2}_\infty}.
\]
as long as $\delta>0$ is small enough so that $2\alpha+\beta-2\delta > 2$. Applying Lemma~\ref{lem:modified paraproduct exp} and twice Lemma~\ref{lem:periodic cutoff bound} we can also replace $\Pi_N ( (\Pi_N (u \mpara X)) \reso \xi)$ with $\Pi_N ( (\Pi_N (u \para X)) \reso \xi)$, so that it remains to control $\Pi_N ( (\Pi_N (u \para X) \reso \xi) - c_N u)$. So far we only reproduced the calculations of Section~\ref{sec:paracontrolled}. But now we cannot simply continue in the same way because we do not have a good enough control of $\Pi_N$,
and in particular it is not true that $\Pi_N (u \mpara X)$ is paracontrolled by $X$ (at least not allowing for uniform bounds in $N$). In~{\cite{Gubinelli2014KPZ}} an approach was developed to tackle this problem and it turns out to be sufficient to control a certain random operator: Set
\[ C_N (u, X, \xi) = (\Pi_N (u \para  X)) \reso  \xi - u ( X
   \reso  \xi) \]
and
\begin{align}\label{eq:random operator}\nonumber
   A_N (u) & = \Pi_N (C_N (u, X, \xi) - C (u, X, \xi) )= \Pi_N ((\Pi_N (u \para X)) \reso \xi - (u \para X) \reso \xi) \\
   & = \Pi_N((( \Pi_N - 1) (u \para X)) \reso \xi).
\end{align}
Then we can expand
\[
   \Pi_N ( (\Pi_N (u \para X) \reso \xi) - c_N u) = A_N(u) + \Pi_N (u (X \reso \xi - c_N)),
\]
and the second term on the right hand side can be controlled using Lemma~\ref{lem:periodic cutoff bound} by
\[
   \| \Pi_N (u (X \reso \xi - c_N))\|_{\mathcal M_T^{\alpha/2} \CC^{\alpha + \beta - 2}_1} \lesssim \| u \|_{\mathcal M^\alpha_T \CC_1^\alpha} \| X \reso \xi - c_N \|_{\CC^{2\alpha-2}_\infty}.
\]
So if we assume that $A_N$ is a bounded linear operator from $\CC_1^\alpha$ to $\CC_1^{2\alpha-2}$, then all the terms on the right hand side of~\eqref{eq:LN usharp} are under control and from here it is straightforward to show the convergence of $u_N$ to the solution $u$ of Proposition~\ref{prop:continuous pam} as long as $(u_0^N, \xi_N, X_N \reso \xi_N - c_N, A_N) \Rightarrow (u_0, \xi, X\diamond \xi, 0)$ in $\CC_1^0 \times \CC^\alpha_\infty \times \CC^{2\alpha-2}_\infty \times L(\CC_1^\alpha, \CC_1^{2\alpha-2})$, where $L(X,Y)$ denotes the space of bounded linear operators from $X$ to $Y$; see~\cite{Gubinelli2013, Gubinelli2014KPZ} for similar arguments.

\begin{proposition}\label{prop:deterministic discrete limit}
   Assume that $(u_0^N, \xi_N, X_N \reso \xi_N - c_N, A_N)$ converges to $(u_0, \xi, X\diamond \xi, 0)$ in $\CC_1^0 \times \CC^\alpha_\infty \times \CC^{2\alpha-2}_\infty \times L(\CC_1^\alpha, \CC_1^{2\alpha-2})$. Then the solution $u_N$ to
   \[
      \LL_N u_N = \Pi_N( u_N  \xi_N) - c_N u_N, \hspace{2em} u_N (0) = u_N^0,
   \]
   converges in $C([0,T], \CC^0_1)$ to the solution $u$ of
   \[
      \LL u = u \diamond \xi, \qquad u(0) = u_0.
   \]
\end{proposition}

\section{Convergence of the potential}\label{sec:martingales}

To complete the proof of Theorem~\ref{thm:main}, it remains to show that the conditions of Proposition~\ref{prop:deterministic discrete limit} are satisfied under our assumptions (H$_{\mathrm{rw}}$), (H$_{\mathrm{mart}}$) and (H$_{\mathrm{init}}$). This will be achieved in this section, which can be seen as the main technical contribution of the paper, with the help of multiple stochastic integrals.

\subsection{Martingale central limit theorem and convergence to the white noise}

The potential is given by $\xi_N= \varepsilon^{-1} \mathcal{E}_N \eta_N (\cdot / \varepsilon)$, and therefore
\[
   \CF \xi_N (k) =\1_{| k |_{\infty} < N / 2} \varepsilon^{-1} \CF_{\T_N} \eta_N (k)  = \1_{| k |_{\infty} < N / 2} \varepsilon \sum_{| \ell |_{\infty} < N / 2} e^{- i \langle k, \varepsilon \ell  \rangle} \eta_N (\ell) .
\]
To prove the convergence of $\xi_N$ to the white noise $\xi$ in distribution
in $\CS'$, it suffices to show that
\[
   \left( \CF \xi_N (k_1), \ldots, \CF \xi_N (k_m) \right) \longrightarrow \left( \CF \xi (k_1), \ldots, \CF \xi (k_m) \right)
\]
in distribution in $\C^m$, for all $(k_1, \ldots k_m) \in
\R^m$. Using the Cram{\'e}r-Wold theorem we can restrict
ourselves to studying the convergence of linear combinations of the Fourier modes, which are of the form $\varepsilon \sum_{| \ell |_{\infty} < N / 2} (\varphi (\varepsilon \ell) + i \psi (\varepsilon \ell)) \eta_N (\ell)$ for suitable real valued, smooth, and bounded functions $\varphi, \psi$.
Applying the Cram{\'e}r-Wold theorem once more, we see that it suffices to study the
convergence of
\[
   S_N = \varepsilon \sum_{| \ell |_{\infty} < N / 2} \varphi (\varepsilon \ell) \eta_N (\ell) = \varepsilon \sum_{k=0}^{N^2-1} \varphi (\varepsilon \zeta(k)) \eta_N(\zeta(k)),
\]
where we recall that $\zeta \colon \{0,\dots, N^2-1\} \to (-N/2,N/2)^2$ is the enumeration under which $\eta_N$ is a martingale. Observe that under (H$_{\mathrm{mart}}$) we have
\begin{align*}
   \lim_{N \to \infty} \sum_{k=0}^{N^2-1} \E[|\varepsilon \varphi (\varepsilon \zeta(k)) \eta_N(\zeta(k))|^2 | \eta_N(\zeta(0), \dots, \eta_N(\zeta(k-1))] & = \lim_{N \to \infty} \varepsilon^2 \sum_{k=0}^{N^2-1} \varphi^2(\varepsilon \zeta(k)) \\
   & = \int_{\T^2} \varphi^2(x) \dd x.
\end{align*}
So by the martingale central limit theorem, \cite{Brown1971}, Theorem~1, it follows that $(S_N)$ converges in distribution to a centered normal variable with variance $\int_{\T^2} \varphi^2(x) \dd x$ provided that we can show
\[
   \lim_{N\to \infty}\sum_{k=0}^{N^2-1} \E[| \varepsilon \varphi (\varepsilon \zeta(k)) \eta_N (\zeta(k))|^2 \1_{|\varepsilon \varphi (\varepsilon \zeta(k)) \eta_N (\zeta(k))| > \delta}] = 0
\]
for all $\delta > 0$. But since by assumption (H$_{\mathrm{mart}}$) the fourth moment of $\eta_N(\ell)$ is uniformly bounded in $N$ and $\ell$, this convergence is easily shown by an application of the Cauchy-Schwarz inequality and the dominated convergence theorem. In conclusion, we have shown the following result.

\begin{lemma}\label{lem:white noise convergence}
   Assume that $(\eta_N (k) : k \in  (-N/2,N/2)^2)$ satisfies (H$_{\mathrm{mart}}$). Then
   \[
     \xi_N (x) = \varepsilon (2 \pi)^{- 2} \sum_{|
     k |_{\infty}, | \ell |_{\infty} < N / 2} e^{i \langle k, x - \varepsilon \ell \rangle} \eta_N (\ell), \qquad x \in \T^2,
  \]
  converges in distribution in $\CS' (\T^2)$ to the white noise on $\T^2$.
\end{lemma}

\begin{remark}
   Of course, the analogous statement holds in $\T^d$ for any $d$.
\end{remark}

\subsection{Multiple stochastic integrals and tightness in Besov spaces}

To derive tightness estimates for the area term $X_N \diamond \xi_N$ it will be useful to rewrite it as a second order stochastic integral with respect to $(\eta_N)$, which is an idea that was inspired by~{\cite{Mourrat2014}}, Lemma~4.1. For the general discussion of multiple stochastic integrals we will take our index set to be $\N$ rather than $(-N/2,N/2)^2$ in order to facilitate the presentation.

Let $(\eta (k) : k = 0, 1, \ldots)$ be a sequence of martingale differences, let $n \in \N$ and let $f \in \ell^2
(\N^n)$ with $f (k_1, \ldots, k_n) = 0$ whenever $k_i = k_j$ for some
$i \neq j$. Then we define
\[ I_n (f) = \sum_{k_1, \ldots, k_n \in \N} f (k_1, \ldots, k_n) \eta
   (k_1) \cdot \ldots \cdot \eta (k_n) . \]
By definition we have $I_n (f) = I_n (\tilde{f})$, where
\[ \tilde{f} (k_1, \ldots, k_n) = \frac{1}{n!} \sum_{\sigma \in \mathcal{S}_n}
   f (\sigma (k_1), \ldots, \sigma (k_n)), \]
is the symmetrization of $f$ with $\mathcal{S}_n$ denoting the group of permutations of $\{ 1, \ldots, n
\}$. Moreover,
\[ I_n (\tilde{f}) = n! \sum_{k_1 < \ldots < k_n} \tilde{f} (k_1, \ldots, k_n)
   \eta (k_1) \cdot \ldots \cdot \eta (k_n) . \]
This representation is nice, because now $I_n (\tilde{f})$ is given as a sum of martingale
increments: we have
\[ I_n (\tilde{f}) = n! \sum_{k_n} I_{n - 1} (\tilde{f} (\cdot |k_n)) \eta
   (k_n), \]
with
\[ \tilde{f} (\cdot |k_n) (k_1, \ldots, k_{n - 1}) = \tilde{f} (k_1, \ldots,
   k_n) \]
whenever $k_1 < \ldots < k_{n - 1} < k_n$, and 0 otherwise. Therefore, $I_n
(\tilde{f})$ is a martingale transform of $\left( \sum_{k \le \cdot}
\eta (k) \right)$.

\begin{proposition}\label{prop:hypercontractivity}
   Let $p \ge 2$, $n \in \N$ and $M>0$ and let $(\eta (k) : k = 0, 1, \ldots)$ be a sequence of martingale differences with
   \[
      \E[|\eta(k)|^{p}|\eta(0), \dots, \eta(k-1)] \le M
   \]
   for all $k$. Then we have for any $f \in \ell^2(\N^n)$ with $f (k_1, \ldots, k_n) = 0$ whenever $k_i = k_j$ for some
$i \neq j$
  \[
     \| I_n (f) \|_{L^p (\Omega)}^p =\E [| I_n (f) |^p] \lesssim \Big( \sum_{k_1, \ldots, k_n} | f (k_1, \ldots, k_n) |^2 \Big)^{p / 2} M^n = \| f \|_{\ell^2 (\N^n)}^p M^n.
  \]
\end{proposition}

\begin{proof}
  Let us start with $n = 1$. In that case the discrete time Burkholder-Davis-Gundy inequality gives
  \begin{align*}
      \E [| I_1 (f) |^p] & \simeq \E \Big[ \Big| \sum_k | f(k) |^2 | \eta (k) |^2 \Big|^{p / 2} \Big] = \Big\| \sum_k | f(k) |^2 | \eta (k) |^2 \Big\|_{L^{p/2}(\Omega)}^{p/2} \le \Big( \sum_k | f(k) |^2 \| \eta(k) \|_{L^p(\Omega)}^{1/2} \Big)^{p/2} \\
      & \le \Big( \sum_k | f(k) |^2 \Big)^{p/2} M,
  \end{align*}
  where we used that $p \ge 2$ and therefore Minkowski's inequality applies. Assume now the claim is shown for $n - 1$. Then we apply again the Burkholder-Davis-Gundy inequality and Minkowski's inequality to get
  \begin{align*}
     \E [| I_n (f) |^p] =\E \Big[ \Big| \sum_{k_n} I_{n - 1} (\tilde{f} (\cdot |k_n)) \eta (k_n) \Big|^p \Big] & \lesssim \E \Big[ \Big( \sum_{k_n} | I_{n - 1} (\tilde{f} (\cdot |k_n)) |^2 | \eta (k_n) |^2 \Big)^{p / 2} \Big] \\
     &  \le \Big(\sum_{k_n} \E[| I_{n - 1} (\tilde{f} (\cdot |k_n)) |^p | \eta (k_n) |^p]^{2/p} \Big)^{p/2} \\
     & \le \Big(\sum_{k_n} \E[| I_{n - 1} (\tilde{f} (\cdot |k_n)) |^p ]^{2/p} \Big)^{p/2} M.
  \end{align*}
  The induction hypothesis now yields
  \begin{align*}
     \sum_{k_n} \E[| I_{n - 1} (\tilde{f} (\cdot |k_n)) |^p]^{2/p} & \lesssim \sum_{k_n} \Big( \Big(\sum_{k_1, \ldots, k_{n-1}} | \tilde f (k_1, \ldots, k_{n-1} | k_n) |^2\Big)^{p/2} M^{n-1} \Big)^{2 / p} \\
     & \le \sum_{k_1,\dots,k_n} |f(k_1, \dots, k_n)|^2 M^{2(n-1) / p},
  \end{align*}
  from where the claim readily follows.
\end{proof}

\begin{remark}
  We assumed that $f$ is real valued, but of course Proposition~\ref{prop:hypercontractivity} extends to complex
  valued $f$ by writing $f = f_1 + i f_2$ for real valued $f_1, f_2$ and then
  applying Proposition~\ref{prop:hypercontractivity} for $f_1$ and $f_2$ separately. 
\end{remark}

The following simple observation will be used many times, which is why we formulate it as a lemma.

\begin{lemma}\label{lem:fourier diagonalization}
  Let $N$ be odd and let $k \in
  \Z^d$. Then
  \[
     \sum_{| \ell |_{\infty} < N / 2} e^{i \langle k, \varepsilon \ell \rangle} = \prod_{j = 1}^d e^{i k_j (- N / 2 + 1 / 2)} \sum_{\ell_j = 0}^{N - 1} e^{i \varepsilon k_j \ell_j } =  \prod_{j = 1}^2 e^{i k_j (- N / 2 + 1 / 2)} (N\1_{k_j = 0}) = N^d \1_{k = 0}.
  \]
\end{lemma}

\begin{corollary}\label{cor:fourier stochastic integral}
   Assume that $\eta_N$ satisfies (H$_{\mathrm{mart}}$) and define $\xi_N= \varepsilon^{-1} \mathcal{E}_N \eta_N (\cdot / \varepsilon)$. Let $n \in \{1,2\}$ and $f \colon (\Z^2)^n \to \C$. Then
  \begin{align*}
     &\E\Big[ \Big| \sum_{k_1, \dots, k_n \in E_N} f(k_1,\dots,k_n) (\CF \xi_N(k_1) \dots \CF \xi_N(k_n) - \E[\CF \xi_N(k_1) \dots \CF \xi_N(k_n)]) \Big|^{p/n} \Big] \\
     &\hspace{50pt} \lesssim \Big( \sum_{k_1, \dots, k_n \in E_N}  |f (k_1, \ldots, k_n)|^2 \Big)^{p / (2n)} M,
  \end{align*}
  where we introduced the notation
  \[
     E_N = \{ k \in \Z^2: |\Z|_\infty < N/2\}.
  \]
\end{corollary}

\begin{remark}
   As the notation suggests we expect a similar bound (involving subtractions of more complicated corrector terms than only the expectation) to hold at least for all $n \le p/2$. But since here we only need the cases $n=1,2$ for which the proof is relatively simple, we do not study the general case.
\end{remark}

\begin{proof}
   We prove the claim for $n=2$, the case $n=1$ follows from similar but simpler arguments. We have
   \begin{align}\label{eq:fourier stochastic integral pr1} \nonumber
      & \E\Big[ \Big| \sum_{k_1, k_2 \in E_N} f(k_1,k_2) (\CF \xi_N(k_1) \CF \xi_N(k_2) - \E[\CF \xi_N(k_1) \CF \xi_N(k_2)]) \Big|^{p/2} \Big] \\ \nonumber
      &\hspace{50pt} = \E\Big[ \Big| \sum_{k_1,k_2 \in E_N}  f(k_1,k_2) \sum_{\ell_1, \ell_2 \in E_N} \varepsilon^2 e^{- i \langle k_1, \varepsilon \ell_1 \rangle - i \langle k_2, \varepsilon \ell_2 \rangle} (\eta_N(\ell_1) \eta_N(\ell_2) - \delta_{\ell_1,\ell_2}) \Big|^{p/2} \Big] \\ \nonumber
      &\hspace{50pt} \lesssim \E\Big[ \Big| \sum_{\ell_1 \neq \ell_2} \Big(\sum_{k_1, k_2 \in E_N} f(k_1,k_2) \varepsilon^2 e^{- i \langle k_1, \varepsilon \ell_1 \rangle} e^{- i \langle k_2, \varepsilon \ell_2 \rangle} \Big) \eta_N(\ell_1) \eta_N(\ell_2) \Big|^{p/2} \Big] \\
      &\hspace{80pt} +  \E\Big[ \Big| \sum_{\ell \in E_N}\Big(\sum_{k_1, k_2 \in E_N} f(k_1,k_2) \varepsilon^2 e^{- i \langle k_1 + k_2, \varepsilon \ell \rangle} \Big) (\eta_N(\ell)^2 - 1)\Big|^{p/2} \Big],
   \end{align}
   and assumption~(H$_{\mathrm{rw}}$) implies that $(\eta_N(\zeta(\ell))^2 - 1)_{\ell =0,\dots, N^2-1}$ is a martingale with
   \[
      \E[|\eta_N(\zeta(\ell))^2 - 1|^{p/2} | (\eta_N(\zeta(0))^2 - 1), \dots, (\eta_N(\zeta(\ell-1))^2 - 1)] \le M,
   \]
   so Proposition~\ref{prop:hypercontractivity} yields
   \begin{align*}
      &\E\Big[ \Big| \sum_{\ell \in E_N}\Big(\sum_{k_1, k_2 \in E_N} f(k_1,k_2) \varepsilon^2 e^{- i \langle k_1 + k_2, \varepsilon \ell \rangle} \Big) (\eta_N(\ell)^2 - 1) \Big|^{p/2} \Big] \\
      &\hspace{50pt} \lesssim \Big(\sum_{\ell \in E_N}\Big|\sum_{k_1, k_2 \in E_N} f(k_1,k_2) \varepsilon^2 e^{- i \langle k_1 + k_2, \varepsilon \ell \rangle} \Big|^2\Big)^{p/4} M.
   \end{align*}
   Similarly we get for the first term on the right hand side of~\eqref{eq:fourier stochastic integral pr1}
   \begin{align*}
      &\E\Big[ \Big| \sum_{\ell_1 \neq \ell_2} \Big(\sum_{k_1, k_2 \in E_N} f(k_1,k_2) \varepsilon^2 e^{- i \langle k_1, \varepsilon \ell_1 \rangle} e^{- i \langle k_2, \varepsilon \ell_2 \rangle} \Big) \eta_N(\ell_1) \eta_N(\ell_2) \Big|^{p/2} \Big] \\
      &\hspace{50pt} \lesssim \Big(\sum_{\ell_1 \neq \ell_2} \Big| \sum_{k_1, k_2 \in E_N} f(k_1,k_2) \varepsilon^2 e^{- i \langle k_1, \varepsilon \ell_1 \rangle} e^{- i \langle k_2, \varepsilon \ell_2 \rangle} \Big|^2 \Big)^{p/4} M,
   \end{align*}
   and in conclusion
   \begin{align*}
      & \E\Big[ \Big| \sum_{k_1, k_2 \in E_N} f(k_1,k_2) (\CF \xi_N(k_1) \CF \xi_N(k_2) - \E[\CF \xi_N(k_1) \CF \xi_N(k_2)])  \Big|^{p/2} \Big] \\
      &\hspace{50pt} \lesssim \Big(\sum_{\ell_1, \ell_2 \in E_N} \Big| \sum_{k_1, k_2 \in E_N} f(k_1,k_2) \varepsilon^2 e^{- i \langle k_1, \varepsilon \ell_1 \rangle} e^{- i \langle k_2, \varepsilon \ell_2 \rangle} \Big|^2 \Big)^{p/4} M.
   \end{align*}
   Now we expand the double sum $|\sum_{k_1,k_2} a(k_1,k_2)|^2 = \sum_{k_1,k_1', k_2, k_2'} a(k_1,k_2) a(k_1,k_2)^\ast$, where $(\cdot)^\ast$ denotes the complex conjugate, and then apply Lemma~\ref{lem:fourier diagonalization} to collapse the big sum to the diagonals $k_1 = k_1'$ and $k_2 = k_2'$, which leads to
   \[
     \E\Big[ \Big| \sum_{k_1, k_2 \in E_N} f(k_1,k_2) (\CF \xi_N(k_1) \CF \xi_N(k_2) - \E[\CF \xi_N(k_1) \CF \xi_N(k_2)]) \Big|^{p/2} \Big] \lesssim \Big( \sum_{k_1, k_2 \in E_N} |f(k_1,k_2)|^2 \Big)^{p/4} M,
   \]
   and this concludes the proof.
\end{proof}

With the help of this corollary the tightness proof for the potential is
quite straightforward.

\begin{lemma}\label{lem:wn tightness}
   Assume that $\eta_N$ satisfies (H$_{\mathrm{mart}}$). Define $\xi_N= \varepsilon^{-1} \mathcal{E}_N \eta_N (\cdot / \varepsilon)$.  Then we have for all $\gamma < - 1 - 2 / p$
  \begin{equation}
    \label{eq:wn moment bound} \sup_N \E [\| \xi_N \|_{\CC^\gamma_\infty}^p]
    \lesssim M .
  \end{equation}
  In particular, $(\xi_N)$ converges to the white noise $\xi$ in distribution in
  $\CC^{\gamma}_\infty$ for all $\gamma < - 1 - 2 / p$.
\end{lemma}

\begin{proof}
  We already established the convergence of $(\xi_N)$ to the white noise in
  Lemma~\ref{lem:white noise convergence}. Once we establish~(\ref{eq:wn moment bound}), we get tightness of $(\xi_N)$ in
  $\CC^{\gamma'}_\infty$ for all $\gamma' < \gamma < - 1 - 2 / p$, from where the
  claimed convergence follows. But by the Besov embedding theorem,
  Lemma~\ref{lem:besov embedding}, we have
  $\| \xi_N \|_{\CC^{\gamma}_\infty} \lesssim \| \xi_N \|_{B^{\gamma + 2 / p}_{p,
  p}}$. Let us write $\beta = \gamma + 2 / p$. Then
  \[ \E [\| \xi_N \|_{B^{\beta}_{p, p}}^p] = \sum_{j \ge - 1}
     2^{j p \beta} \E [\| \Delta_j \xi_N \|_{L^p}^p] = \sum_{j
     \ge - 1} 2^{j p \beta} \int_{\T^2} \E [| \Delta_j
     \xi_N (x) |^p] \mathd x,
  \]
  and from Corollary~\ref{cor:fourier stochastic integral} we get
  \[
     \E [| \Delta_j \xi_N (x) |^p] = \E \Big[ \Big| \sum_{k \in E_N} (2\pi)^{-2} e^{i \langle k, x \rangle}\rho_j(k) \CF \xi_N(k) \Big|^p \Big] \lesssim \Big( \sum_{k \in E_N} \rho_j (k)^2 \Big)^{p/2} M \lesssim 2^{j p} M,
  \]
  which multiplied with $2^{jp\beta}$ is summable in $j$ whenever $\beta < -1$.
\end{proof}

Next, we need to study the convergence of $(X_N)$ and of $(X_N
\diamond \xi_N)$. For $X_N$ we have $\CF X_N (k) =\1_{k \neq 0} \CF
\xi_N (k) / (f (\varepsilon k) | k |^2)$, from where it easily follows that $
(X_N, \xi_N)$ converges jointly in distribution to $(X, \xi)$. Moreover, since
$\xi_N$ is spectrally supported on the set $(- N / 2, N / 2)^2$ where $f(\varepsilon k) \ge c_f > 0$, we get
\[ \| X_N \|_{\gamma + 2} \lesssim \|  \xi_N - (2\pi)^{-2} \CF \xi_N(0)\|_{\gamma} \lesssim
   \| \xi_N \|_{\gamma}, \]
from where we get the tightness of $(X_N)$ in $\CC^{\gamma+2}_\infty$ for all $\gamma < -1-2/p$. The term $X_N \diamond \xi_N$ is
more tricky. There are limit theorems for polynomials of i.i.d.
variables, see for example~{\cite{Janson1997, Mossel2010, Caravenna2013}}, and it should be possible to generalize them to the case of martingale increments.
However, here we can simply use a relatively cheap diagonal sequence argument to combine the identification of
the limit of $(X_N \diamond \xi_N)$ with the proof of its tightness. This is inspired by Mourrat and Weber~{\cite{Mourrat2014}}, Theorem~6.2.

\begin{lemma}\label{lem:area convergence}
  In the setting of Lemma~\ref{lem:wn tightness} define $\CF X_N (k) = \1_{\{k \neq 0 \}} \CF \xi_N (k) / (f(\varepsilon k) | k |^2)$ and set
  \[
     X_N \diamond \xi_N = X_N \reso \xi_N - \tilde{c}_N
  \]
  with $\tilde{c}_N = (2 \pi)^{- 2} \sum_{| k |_{\infty} < N / 2} \frac{\1_{k \neq 0}}{f (\varepsilon k) | k |^2}$. Then we have for all $\gamma < - 4 / p$
  \[
     \sup_N \E [\| X_N \diamond \xi_N \|_{\CC^\gamma_\infty}^{p/2}] \lesssim M. 
  \]
  Moreover, with $c_K = (2 \pi)^{- 2} \sum_{| k |_{\infty} < K / 2}\frac{\1_{k \neq 0}}{| k |^2}$ we get for all $N > K^2$ and all $\gamma \in (-1-4/p,-4/p)$
  \[
    \sup_N \E [\| X_N \diamond \xi_N - (\mathcal{P}_K X_N \reso \mathcal{P}_K \xi_N - c_K) \|_{\CC^\gamma_\infty}^{p/2}] \lesssim K^{\gamma + 4 / p} M,
  \]
  where $\mathcal{P}_K u = \CF^{-1} (\1_{(-K/2,K/2)^2} \CF u)$.
\end{lemma}

\begin{proof}

   Applying Lemma~\ref{lem:fourier diagonalization} together with the fact that $\sum_{| i - j | \le 1} \rho_i (k) \rho_j (k) = 1$, we get for any $x \in \T^2$
   \begin{align*}
      \tilde{c}_N & = \varepsilon^2 (2 \pi)^{- 4} \sum_{| i - j | \le 1} \sum_{| \ell |_{\infty}, | k_1 |_{\infty}, | k_2 |_{\infty} < N / 2} \frac{\1_{k_1 \neq 0}}{f (\varepsilon k_1) | k_1 |^2} e^{i \langle k_1 + k_2, x - \varepsilon \ell \rangle} \rho_i (k_1) \rho_j (k_2)  = \E[(X_N \reso \xi_N)(x)].
   \end{align*}
   Similarly we obtain $\E[\Delta_q (X_N \reso \xi_N)(x)] = 0$ for $q\ge 0$, and therefore $\Delta_q \tilde c_N = \E[\Delta_q (X_N \reso \xi_N)(x)]$ for all $q \ge -1$ and all $x \in \T^2$. So if we write $\Delta_q (X_N \reso \xi_N)(x) =  \sum_{k_1, k_2 \in E_N} a^N_{q,x} (k_1, k_2) \CF \xi_N (k_1) \CF \xi_N (k_2)$, then
   \begin{align*}
      \E [ | \Delta_q (X_N \diamond \xi_N) (x) |^{p/2}] & = \E\Big[\Big|\sum_{k_1, k_2 \in E_N} a^N_{q,x} (k_1, k_2) ( \CF \xi_N (k_1) \CF \xi_N (k_2) - \E[\CF \xi_N (k_1) \CF \xi_N (k_2)]) \Big|^{p/2} \Big] \\
      & \lesssim \Big( \sum_{k_1,k_2 \in E_N} |a^N_{q,x} (k_1, k_2)|^2 \Big)^{p/4} M
   \end{align*}
   by Corollary~\ref{cor:fourier stochastic integral}, and
   \begin{align}\label{eq:area convergence pr1} \nonumber
      \sum_{k_1,k_2 \in E_N} |a^N_{q,x} (k_1, k_2)|^2 & = (2 \pi)^{- 4} \sum_{k_1,k_2 \in E_N} \rho_q (k_1 + k_2)^2 \frac{\1_{k_1 \neq 0}}{f (\varepsilon k_1)^2 | k_1 |^4} \Big( \sum_{| i - j | \le 1} \rho_i (k_1) \rho_j (k_2) \Big)^2 \\
       &\lesssim \sum_{k_1,k_2 \in E_N} \1_{|k_1 + k_2 | \sim 2^q} \frac{\1_{k_1 \neq 0}}{c_f^2 | k_1 |^4} \1_{| k_1 | \sim | k_2 |} \lesssim \sum_{| k_1 |_{\infty} \gtrsim 2^q} 2^{2 q} \frac{1}{| k_1 |^4} \lesssim 1.
   \end{align}
  
  If instead of $X_N \diamond \xi_N$ we are considering $X_N \diamond \xi_N - (\mathcal{P}_K X_N \reso
  \mathcal{P}_K \xi_N - c_K)$, then we have two contributions: the first one, 
  \[
    X_N \reso \xi_N - \mathcal{P}_K X_N \reso \mathcal{P}_K \xi_N - \E[X_N \reso \xi_N - \mathcal{P}_K X_N \reso \mathcal{P}_K \xi_N],
  \]
  can be bounded as before: We obtain an additional factor $(\1_{| k_1 | > K / 2} +\1_{| k_1 | < K / 2} \1_{| k_2 | > K / 2})^2$ in equation~(\ref{eq:area convergence pr1}), resulting in the improved upper bound
  \begin{equation}\label{eq:area convergence pr3}
    \sum_{k_1,k_2 \in E_N} |a^N_{q,x} (k_1, k_2)|^2 \lesssim 2^{2 q \lambda} K^{- 2 \lambda}
  \end{equation}
  for all $\lambda \in [0, 1]$. The second contribution (which only appears in the block $q=-1$) is
  \begin{align}\label{eq:area convergence pr4} \nonumber
     | c_K -\E [\mathcal{P}_K X_N \reso \mathcal{P}_K \xi_N] | & = \Big| (2 \pi)^{- 2} \sum_{| k |_{\infty} < K / 2} \left( \frac{\1_{k \neq 0}}{| k|^2} - \frac{\1_{k \neq 0}}{f (\varepsilon k) | k |^2} \right) \Big| \lesssim \sum_{| k |_{\infty} < K / 2} \frac{\1_{k \neq 0}}{| k|^2} \frac{| f (\varepsilon k) - f (0) |}{c_f} \\ 
      & \lesssim N^{- 1} \sum_{| k |_{\infty} < K / 2} \frac{\1_{k \neq 0}}{| k |} \lesssim N^{- 1} K \lesssim K^{- 1},
  \end{align}
  where we used that $\varepsilon = 2 \pi / N$ and that $N \ge K^2$.
  
  The bound (\ref{eq:area convergence pr1}) now gives us
  \[
     \E [\| X_N \diamond \xi_N \|_{B^{\beta}_{p/2, p/2}}^{p/2}] \lesssim \sum_{q \ge - 1} 2^{q \beta p/2} M \lesssim M
  \]
  for all $\beta < 0$. Combining instead \eqref{eq:area convergence pr3} and \eqref{eq:area convergence pr4}, we get for all $\lambda_q \in [0, 1]$
  \[
     \E [\| X_N \diamond \xi_N - (\mathcal{P}_K X_N \reso \mathcal{P}_K \xi_N - \tilde{c}_K) \|_{B^{\beta}_{p/2, p/2}}^{p/2}] \lesssim \sum_{q \ge - 1} 2^{q \beta p/2 } 2^{q \lambda_q p/2} K^{- \lambda_q p/2} M.
  \]
  Setting $\lambda_q = 1$ for $2^q \le K$ and $\lambda_q = 0$ for $2^q > K$, we see that the right hand side is bounded by $\lesssim K^{p \beta} M$ whenever $\beta \in (- 1, 0)$. The claim now follows from the Besov embedding theorem, Lemma~\ref{lem:besov embedding}.
\end{proof}

\begin{corollary}\label{cor:stoch data convergence}
  Make assumptions (H$_{\mathrm{mart}}$), (H$_{\mathrm{rw}}$) and (H$_{\mathrm{init}}$) and let $\alpha < 1 - 2 /
  p$. Then $(u^N_0, \xi_N, X_N, X_N \diamond \xi_N)$ converges jointly in
  distribution in $\CC^0_1 \times \CC^{\alpha - 2}_\infty \times \CC^{\alpha}_\infty \times \CC^{2 \alpha -
  2}_\infty$ to $(u_0, \xi, X, X \diamond \xi)$, where $\xi$ is a white noise, $\CF X (0) = 0$, $\CF X (k) = \CF \xi (k) / | k |^2$ for $k \neq 0$, and
  \[
     X \diamond \xi = \lim_{K \rightarrow \infty} \mathcal{P}_K X \reso \mathcal{P}_K \xi -  c_K,
  \]
  for which the convergence was established in~{\cite{Gubinelli2013}}, Lemma~5.7.
\end{corollary}

\begin{proof}
  The moment bounds that we derived (or assumed in the case of $u_0^N$) imply the joint tightness of $(u_0^N, \xi_N, X_N, X_N \diamond \xi_N)$ in $\CC^0_1 \times \CC^{\alpha - 2}_\infty \times \CC^{\alpha}_\infty \times \CC^{2 \alpha - 2}_\infty$. The identification of the limit points is trivial, except in the case of $(X_N \diamond \xi_N)$. But since
   \[
      \CC^{\alpha - 2}_\infty \times \CC^{\alpha}_\infty \ni (\varphi, \psi) \mapsto Q_K(\varphi,\psi) := \mathcal{P}_K \varphi \reso \mathcal{P}_K  \psi - c_K \in \CC^{2 \alpha - 2}_\infty
   \]
  is a continuous function, we get that for fixed $K$ the sequence $(Q_K (X_N, \xi_N))$ converges to $Q_K (X,\xi)$ in distribution. Now we simply estimate
  \begin{align*} 
     \| X_N \diamond \xi_N - X \diamond \xi \|_{\CC^{2 \alpha - 2}_\infty} & \le   \| X_N \diamond \xi_N - Q_K (X_N \diamond \xi_N) \|_{\CC^{2 \alpha - 2}_\infty} + \|    Q_K (X_N, \xi_N) - Q_K (X, \xi) \|_{\CC^{2 \alpha - 2}_\infty} \\
     &\quad + \| Q_K (X, \xi) - X \diamond \xi \|_{\CC^{2 \alpha - 2}_\infty}.
  \end{align*}
  For any $K$ the middle term vanishes as $N \rightarrow \infty$, while by Lemma~\ref{lem:area convergence} the first term satisfies for some $\delta>0$
  \[
     \limsup_{N\to \infty}\E[ \| X_N \diamond \xi_N - Q_K (X_N \diamond \xi_N) \|_{\CC^{2 \alpha - 2}_\infty}^{p/2}] \lesssim K^{-\delta} M,
  \]
  and by definition of $X \diamond \xi$ the third term on the right hand side converges to zero as $K \to \infty$.  \end{proof}

\subsection{Bounds on the random operator}

It remains to bound the random operator
\begin{equation}\label{eq:AN def stoch}
   A_N (u) = \Pi_N[(\Pi_N (u \para X_N)) \reso \xi_N] - \PC_N [(u \para X_N) \reso \xi_N]
\end{equation}
introduced in (\ref{eq:random operator}). Let us write $\psi_{\prec} (k, \ell) = \sum_{j\ge 1} \chi (2^{j-1} k) \rho_j (\ell)$ and $\psi_{\circ} (k, \ell) = \sum_{| i - j | \leqslant 1} \rho_i(k) \rho_j (\ell)$, where we recall that $(\chi,\rho)$ is our dyadic partition of unity. We also write $k_{[12]} = k_1 + k_2$ for $k_1, k_2 \in \Z^2$ and recall that for $k \in \Z^2$
\[
   (k^N)_r = \arg \min \{ |\ell| : \ell = k_r + j N \text{ for some } j \in \Z\} \in (-N/2, N/2), \qquad r = 1,2.
\]

\begin{lemma}[\cite{Gubinelli2014KPZ}, Lemma 10.5]
   The operator $A_N$ defined in~\eqref{eq:AN def stoch} is given by
   \begin{equation}\label{eq:random operator representation}
      A_N (u) (x) = \sum_{i,j \ge -1} \Delta_j (A_N (\Delta_i u)) (x) = \sum_{i, j \ge -1} \int_{\T^2} g^N_{i, j} (x, y) \Delta_i u (y) \mathd y
   \end{equation}
   with
  \begin{equation}\label{eq:gpq fourier transform}
     \CF g^N_{i, j} (x, \cdot) (k) = \sum_{k_1, k_2 \in E_N} \Gamma_{i,j}^N (x ; k, k_1, k_2) \CF X_N (k_1) \CF \xi_N (k_2),
  \end{equation}
  where
  \begin{align*}
     \Gamma_{i,j}^N (x ; k, k_1, k_2) & = (2 \pi)^{- 4} \tilde{\rho}_i (k) \psi_{\prec} (k, k_1)  \times \big[ e^{i \langle (k_{[12]} - k)^N, x \rangle} \rho_j ((k_{[12]} - k)^N) \psi_{\circ}((k_1 - k)^N, k_2)  \\
     &\hspace{140pt} -  e^{i \langle k_{[12]} - k, x\rangle} \rho_j (k_{[12]} - k) \psi_{\circ}(k_1 - k, k_2) \1_{|k_{[12]}-k|_\infty \leqslant N/2}\big],
  \end{align*}
  and where $\tilde{\rho}_i$ is a smooth function supported in an annulus $2^i \CA$ such that $\tilde{\rho}_i \rho_i = \rho_i$. 
\end{lemma}
A similar representation as~\eqref{eq:random operator representation} was derived in~\cite{Gubinelli2014KPZ} for a similar random operator and in our case the proof is exactly the same, which is why we do not reproduce it.

\begin{lemma}\label{lem:random operator bound}
   For $\alpha \in (1/2,1)$ the following convergence holds in probability:
   \[
      \lim_{N \to \infty} \|A_N\|_{L(\CC_1^{\alpha}, \CC_{1}^{2\alpha-2})} = 0.
   \]
\end{lemma}

\begin{proof}
We start by splitting the operator $A_N$ in two parts, $A_N=A_N^1+A_N^2$, where 
\begin{equation}
A_N^1(u)=\sum_{i, j} \int_{\T^2} \mathbb E[g^N_{i, j} (x, y)] \Delta_i u (y) \mathd y
\end{equation}
and 
\begin{equation}
A_N^2(u)=\sum_{i,j} \int_{\T^2} (g^N_{i,j}(x,y)-\mathbb E[g^N_{i, j} (x, y)]) \Delta_i u (y) \mathd y.
\end{equation}
To treat $A^2_N$ let us write $\tilde g^N_{i,j}(x,y) = g^N_{i,j}(x,y)-\mathbb E[g^N_{i, j} (x, y)]$ and observe that 
\begin{align*}
   \| A^2_N(u) \|_{\CC^\beta_1} & \le \| A^2_N (u) \|_{B^{\beta}_{1,1}} \lesssim \sum_{i,j} 2^{j \beta} \int_{\T^2}  \| \tilde g^N_{i, j} (x, y) \|_{L^\infty_y} \mathd x \| \Delta_i u \|_{L^1} \\
   & \le \sum_{i,j} 2^{j \beta} 2^{-i\alpha} \int_{\T^2}  \| \tilde g^N_{i, j} (x, y) \|_{L^\infty_y} \mathd x \| u \|_{\CC^\alpha_1},
\end{align*}
and therefore
\begin{equation}\label{eq:random operator bound pr 1}
   \E[ \| A^2_N \|_{L(\CC^\alpha_1,\CC^{\beta}_{1})}] \lesssim \sum_{i,j} 2^{j \beta} 2^{- i\alpha} \int_{\T^2} \E[\| \tilde g^N_{i, j} (x, y) \|_{L^\infty_y}] \mathd x.
\end{equation}
To control the expectation on the right hand side we apply the trivial bound
\[
   \mathbb E[\|\tilde g^N_{i,j}(x,\cdot)\|_{L^\infty}]\lesssim \sum_k\mathbb E[|\mathscr F(\tilde g^N_{i,j}(x,\cdot))(k)|] \le \sum_k\mathbb E\left[|\mathscr F(\tilde g^N_{i,j}(x,\cdot))(k)|^2\right]^{\frac{1}{2}}.
\]
At this stage let us observe that 
\begin{align*}
        & E\left[|\mathscr F(\tilde g^N_{i,j}(x,\cdot))(k)|^2\right]=\sum_{k_1,k_2 \in E_N}|\Gamma_{i,j}^N (x ; k, k_1, k_2)|^2 \frac{\1_{k_1 \neq 0}}{f(\varepsilon k_1)^2 |k_1|^4} \\
        &\hspace{50pt} \lesssim \sum_{k_1,k_2 \in E_N}\tilde{\rho}^2_i (k) \psi^2_{\prec} (k, k_1) \1_{k_1 \neq 0} |k_1|^{-4}\times \Big| e^{i\langle (k_{[12]} - k)^N, x\rangle } \rho_j ((k_{[12]} - k)^N) \psi_{\circ} ((k_1 -  k)^N, k_2) \\
     &\hspace{210pt} - e^{i\langle k_{[12]} - k, x\rangle } \rho_j (k_{[12]} - k) \psi_{\circ} (k_1 -  k, k_2) \1_{|k_{[12]} - k|\leqslant N/2} \Big|^2,
  \end{align*}
  and the difference on the right hand side is zero unless $|k_1|_\infty \simeq N$ so that $|k_1|^{-4}_\infty \simeq N^{-2+\lambda} |k_1|^{-2-\lambda}_\infty$ for any $\lambda > 0$. Moreover, we only have to sum over $|k_1|_\infty > |k|_\infty$ and the summation over $k_2$ gives $O(2^{2j})$ terms which leads to
  \begin{align*}
     \sum_k \Big( \sum_{k_1,k_2 \in E_N}|\Gamma_{i,j}^N (x ; k, k_1, k_2)|^2 \frac{\1_{k_1 \neq 0}}{f(\varepsilon k_1)^2 |k_1|^4} \Big)^{1/2} & \lesssim \sum_k \big( \1_{2^i,2^j \lesssim N} 2^{2j}  N^{-2+\lambda} \tilde{\rho}^2_i (k) |k|^{-\lambda})^{1/2} \\
     & \lesssim \1_{2^i,2^j \lesssim N} 2^{j}  N^{-1+\lambda/2} 2^{i(1-\lambda/2)}.
  \end{align*}
  Plugging this back into~\eqref{eq:random operator bound pr 1} we get for $\beta > - 1$ and $ \lambda/2 < 1 - \alpha$
  \[
     \E[ \| A^2_N \|_{L(\CC^\alpha_1,\CC^{\beta}_{1})}] \lesssim \sum_{i,j} 2^{j \beta} 2^{- i\alpha}  \1_{2^i,2^j \lesssim N} 2^{j}  N^{-1+\lambda/2} 2^{i(1-\lambda/2)} \lesssim N^{1 + \beta - \alpha}.
  \]
  Taking $\beta = 2 \alpha - 2$ (which is $>-1$ because $\alpha > 1/2$), the claim follows for $A^2_N$.

To handle $A_N^1$ let us remark that 
\begin{align*}
 &\mathbb E[g^N_{i,j}(x,y)]\\
 & =(2\pi)^{-4}\sum_{k}\tilde\rho_i(k)\rho_j(k)\left(\sum_{k_1\in E_N} \psi_{\prec} (k, k_1) \big[ \psi_{\circ}((k_1 - k)^N, k_1) -    \psi_{\circ}(k_1 - k, k_1)\big] \frac{\1_{k_1 \neq 0}}{f(\varepsilon k_1) |k_1|^2}\right)e^{2i\pi\langle y-x,k\rangle} \\
 &=: h_{i,j}(y-x),
\end{align*}
and as before we have 
\[
   \|A^1_N(u)\|_{\CC_{1}^{\beta}}\lesssim\sum_{i,j}2^{j\beta}\|h_{i,j}\ast\Delta_i u\|_{L^1}\lesssim \sum_{i,j}2^{j\beta}\|h_{i,j}\ast\Delta_i u\|_{L^2}.
\]
Now the Parseval identity gives  $\|h_{i,j}\ast\Delta_i u\|^2_{L^2}\simeq\sum_{k}|\CF h_{i,j}(k)|^2|\mathscr F(\Delta_iu)(k)|^2$, and
\begin{align*}
      |\CF  h_{i,j}(k)| = \Big| \sum_{k_1 \in E_N} \tilde{\rho}_i (k) \rho_j (k) \psi_{\prec} (k, k_1) \big[ \psi_{\circ}((k_1 - k)^N, k_1) -    \psi_{\circ}(k_1 - k, k_1)\big] \frac{\1_{k_1 \neq 0}}{f(\varepsilon k_1) |k_1|^2} \Big|,
   \end{align*}
   where we used that $|k|_\infty < N/2$ on the support of $\psi_{\prec} (\cdot, k_1)$. Now we use once more that $\psi_{\circ} ((k_1 -  k)^N, k_1) = \psi_{\circ} (k_1 -  k, k_1)$ unless $|k_1 - k|_\infty > N/2$ and $|k_1| \simeq N$, and that there are at most $N \times |k|_\infty$ values of $k_1$ with $|k|_\infty <|k_1|_\infty< N / 2$ and $| k_1 - k |_\infty > N / 2$. Therefore, the sum over $k$ is bounded by
   \[
     \sum_{k}|\CF h_{i,j}(k)|^2 |\mathscr F(\Delta_iu)(k)|^2\lesssim\1_{i\sim j}\1_{2^{i}\lesssim N}N^{-2} \sum_{|k|\sim 2^{i}}|k|^2_{\infty}|\mathscr F(\Delta_iu)(k)|^2\lesssim \1_{i\sim j}\1_{2^{i}\lesssim N}N^{-2}2^{2i}\|\Delta_i u\|^2_{L^2},
   \]
and now Bernstein's inequality, Lemma~2.1 of~\cite{Bahouri2011}, gives 
\[
\|h_{i,j}\ast\Delta_i u\|
_{L^2}\lesssim  \1_{i\sim j}\1_{2^{i}\lesssim N}N^{-1}2^{i}\|\Delta_i u\|_{L^2}\lesssim \1_{i\sim j}\1_{2^{i}\lesssim N}N^{-1}2^{2i}\|\Delta_i u\|_{L^1}.
\]
So finally we can conclude that 
\[
   \E[\|A^1_N\|_{L(\mathscr C_1^{\alpha},\CC^\beta_{1})}] \lesssim N^{-1}\sum_{i\sim j\lesssim \log_2N}2^{j\beta}2^{i(2-\alpha)},
\]
and taking $\beta = 2\alpha - 2$ we get
\[
  \E[\|A^1_N\|_{L(\mathscr C_1^{\alpha},\CC^\beta_{1})}] \lesssim N^{-1}\sum_{i\lesssim \log_2 N}2^{i\alpha}\lesssim N^{\alpha-1},
\]
which converges to zero as long as $\alpha<1$. This concludes the proof. 
\end{proof}

\section{Invariance principle for semi-discrete random polymer measures}\label{sec:polymer}
In \cite{CannizzaroChouk2015} the authors construct the continuous polymer measure with periodic white noise potential defined formally by
\begin{equation}
\mathbb Q_{T,x}(\mathrm d\omega)=Z^{-1}_{T,x} \exp\Big(\int_0^T\xi(\omega(s))\mathrm d s\Big)\mathbb W_x(\mathrm\omega),\quad Z_{T,x} = \E_{\mathbb W_x}\Big[\exp\Big(\int_0^T\xi(\omega(s))\mathrm d s\Big)\Big],
\end{equation} 
where $\mathbb W_x$ is the Wiener measure on $C([0,T], \T^2)$ starting in $x\in \T^2$. Of course, this formula does not really make sense since $\xi$ is only a Schwartz distribution and not a function and therefore the integral $\int_0^T \xi(\omega(s))\dd s$ is not well defined. However, it was shown in \cite{CannizzaroChouk2015} that replacing the white noise by a mollified version gives a sequence of probability measure $(\mathbb Q^\varepsilon_{T,x})$ which are equivalent to the Wiener measure and that this sequence converges in probability in the weak topology to a measure $\mathbb Q_{T,x}$ which does not depend on the way that we mollified the white noise, and which is almost surely singular with respect to the Wiener measure. Moreover, under the measures $(\mathbb Q_{T,x})_{x \in \T^2}$ the canonical process $(B_t)_{t\in[0,T]}$ on $C([0,T], \T^2)$ is an inhomogeneous strong Markov process with transition function
\begin{equation}\label{eq:polymer-sg}
   K_T(s,t)f(x)=\frac{u^{f,T-t}(t-s,x)}{u^1(T-s,x)}, \qquad 0 \le s \le t \le T,
\end{equation}
where $u^1$ and $u^{f,T-t}$ both satisfy the parabolic Anderson equation with initial condition given respectively by the constant function $1$ and $fu^1(T-t,\cdot)$. More precisely
\[
\partial_t u^{1}=\Delta u^{1}+u^{1} \diamond \xi,\qquad u^{1}(0,x)=1,
\]
and 
\[
\partial_t u^{f,s} =\Delta u^{f,s}+u^{f,s} \diamond \xi,\qquad u^{f,s}(0)=fu^1(s).
\]
Now let us come back to our discrete model and write $(B^N_t)_{t\ge 0}$ for the canonical process on the Skorokhod space $D([0,\infty), \Z_N^2)$ (which is equal to the space of continuous functions from $[0,\infty)$ to $\Z_N^2$ because $\Z_N^2$ is equipped with the discrete topology). We write $\tilde{\P}^N_x$ for the law of the continuous-time random walk with generator $\Delta_{\mathrm{rw}}$, started in $x \in \Z_N^2$. By Donsker's theorem we know that the law of $t \mapsto \varepsilon B^N_{\varepsilon^{-2} t}$ under $\tilde{\P}^N_{\varepsilon^{-1} x}$ converges to the Brownian motion on $\T^2$, started in $x$. Our aim is to derive an analogous result for the semi-discrete polymer measure (time is continuous, space discrete), given by
\[
   \tilde{\Q}^N_{T,x}(\dd \omega) = Z^{-1}_{N,T,x} \exp\left(\int_0^{\varepsilon^{-2}T} \varepsilon \eta_N(\omega(s)) \dd s\right) \tilde{\P}^N_x(\dd \omega),
\]
where $Z^{-1}_{N,T,x}$ is a constant renormalizing the mass of $Q^N_{T,x}$ to $1$. If we denote by $(\P^N_{x})_{x \in \T_N^2}$ the law of the rescaled process $(\varepsilon B^N_{\varepsilon^{-2} t})_{t\ge 0}$ under $\tilde{\P}^N_x$, then the law of $(\varepsilon B^N_{\varepsilon^{-2} t})_{t\in [0,T]}$ under $\tilde{\Q}^N_{T,x}$ is given by
\[
   \Q^N_{T,x} (\dd \omega) = Z^{-1}_{N,T,x} \exp\left(\int_0^{T} \xi_N(\omega(s)) \dd s\right) \P^N_x(\dd \omega),
\]
where now $\omega \in D([0,T], \T_N^2)$ and we recall that $\xi_N = \varepsilon^{-1} \mathcal{E}_N \eta_N(\cdot / \varepsilon)$, which here of course is only evaluated in the points of $\T_N^2$, so in fact there would have been no need to apply the extension operator $\mathcal{E}_N$. Now we claim that if we extend the measure $\mathbb Q^N_{T,x}$ to $\mathcal{B}(D([0,T],\mathbb T^2))$ by setting
\[
   \overline{\mathbb Q}^N_{x,T}(A)=\mathbb Q^N_{x,T}(A\cap D([0,T],\mathbb T^2_N)),
\] 
then $(\overline{\mathbb Q}_{x,T}^N)$ converges in distribution in the weak topology to $\mathbb Q_{x,T}$.

\begin{theorem}\label{thm:polymer}
   Make the assumptions (H$_{\mathrm{rw}}$) and (H$_{\mathrm{mart}}$) and let $T>0$. Let for all $x \in \T^2$ and $N \in \N$ the point $\lfloor x \rfloor_N\in \T_N^2$ be such that $|x- \lfloor x \rfloor_N| \le \varepsilon$. Then the family of probability measures $(\overline{\Q}^N_{T,\lfloor x\rfloor_N})_{x \in \T^2}$ converges jointly in distribution in the weak topology to $(\Q_{T,x})_{x \in \T^2}$.
\end{theorem}

\begin{proof}
   
   As explained in the proof of Theorem~\ref{thm:main} we may assume that $(\xi_N, X_N, X_N \diamond \xi_N, A_N)$ converges in probability to $(\xi, X, X \diamond \xi, 0)$ in $\CC_1^0 \times \CC^\alpha_\infty \times \CC^{2\alpha-2}_\infty \times L(\CC_1^\alpha, \CC_1^{2\alpha-2})$. Let us show that then for any $x \in \T^2$ the measures $(\overline{\Q}^N_{T,\lfloor x\rfloor_N})$ that are constructed from $(\xi_N, X_N, X_N \diamond \xi_N)$ as described above converge in probability to the polymer measure $\Q_{T,x}$ constructed from $(\xi, X, X \diamond \xi)$. For this it suffices to show that every subsequence possesses a subsequence for which the convergence holds, and in this way we may suppose that the data $(\xi_N, X_N, X_N \diamond \xi_N, A_N)$ converges almost surely (because it converges in probability and thus almost surely along a subsequence). Now we simply apply Lemma~\ref{lem:EK} in the appendix with $E=\mathbb T^2$, $E_N = \mathbb T^2_N$ and $\psi_N(x)=x$ for $x\in\mathbb T_N^2$. Moreover, if $Y^N$ denotes the process with law $\mathbb Q^N_{T, \lfloor x\rfloor_N}$, then the law of $X^{N}$ is $\overline{\mathbb Q}^N_{T, \lfloor x \rfloor_N}$. Of course, Lemma~\ref{lem:EK} is only formulated for temporally homogeneous Markov processes, but we can use the standard trick of considering the couple $(Y^N_t,t)_{t\in[0,T]}$ to obtain a temporally homogeneous process. Therefore, we get from Lemma~\ref{lem:EK} in the appendix that $(\overline{\Q}^N_{T,\lfloor x\rfloor_N})$ converges weakly to $\Q_{T,x}$, provided that
\begin{equation}\label{eq:group-conv}
   \lim_{N \to \infty} \|K^N_T(s,t)f-K_T(s,t) f\|_{L^{\infty}(\mathbb T^2_N)} = 0
\end{equation}
for all $f \in C(\mathbb T^2, \R)$ and $0 \le s < t \le T$, where $(K^N_T(s,t))_{0\le s \le t \le T}$ denotes the transition function of $Y^N$. But by the Bayes formula we have for any $f \colon \T_N^2 \to \R$, $0 \le s \le t \le T$, and $x \in \T_N^2$
\[
   \E_{\Q_{T,x}^N}[f(B^N_t) | \mathcal{F}_s] = \frac{\E_{\P_{x}^N}[f(B^N_t) Z^{-1}_{N,T,x} \exp(\int_0^{T} \xi_N(B^N_r) \dd r) | \mathcal{F}_s] }{\E_{\P_{x}^N}[Z^{-1}_{N,T,x} \exp(\int_0^{T} \xi_N(B^N_r) \dd r) | \mathcal{F}_s] },
\]
where $(\mathcal{F}_t)_{t\in [0,T]}$ is the filtration generated by $B^N$. The deterministic factor $Z^{-1}_{N,T,x}$ cancels, as well as the contribution $ \exp(\int_0^{s} \xi_N(B^N_r) \dd r)$ which we can pull out of both conditional expectations. By the Markov property of $B^N$ under $\P^N_x$, the remaining contribution of the denominator is then given by
\[
   \E_{\P^N_{B^N_s}}\Big[\exp\Big(\int_0^{T-s} \xi_N(B^N_r) \dd r\Big)\Big] = u_N^1(T-s,B^N_s),
\]
where we applied the Feynman-Kac formula and where
\[
    \partial_t u_N^1 = \Delta_{\mathrm{rw}}^N  u_N^1 + u_N^1 \xi_N - c_N u^1_N,\qquad u_N^1(0) \equiv 1.
\]
Similarly, the remaining contribution of the numerator is
\begin{align*}
   \E_{\P_{B^N_s}^N} \Big[f(B^N_{t-s}) \exp\Big(\int_0^{T-s} \xi_N(B^N_r) \dd r\Big)\Big] & = \E_{\P_{B^N_s}^N} \Big[ e^{\int_0^{t-s} \xi_N(B^N_r) \dd r} f(B^N_{t-s}) \E_{\P_{B^N_{t-s}}^N} [ e^{\int_0^{T-t} \xi_N(B^N_r) \dd r}]\Big]\\
   & = \E_{\P_{B^N_s}^N} \Big[ e^{\int_0^{t-s} \xi_N(B^N_r) \dd r} f(B^N_{t-s}) u^1_N(T-t,B^N_{t-s})\Big] \\
   & = u^{T-t,f}_N(t-s, B^N_{s}),
\end{align*}
where
\[
    \partial_t u_N^{r,f} = \Delta_{\mathrm{rw}}^N  u_N^{r,f} + u_N^{r,f} \xi_N - c_N  u^{r,f}_N,\qquad u_N^{r,f}(0) = f u^1_N(r).
\]
Consequently the transition function $(K^N_T(s,t))_{0\le s \le t \le T}$ of $B^N$ under $(\Q^N_{T,x})_{x \in \T_N^2}$ is given by
\[
   K^N_T(s,t) f(x) = \frac{u^{T-t,f}_N(t-s, x)}{u_N^1(T-s,x)}.
\]
Combining this with the representation~\eqref{eq:polymer-sg} for $(K_T(s,t))$, the claimed convergence now follows from Proposition~\ref{prop:deterministic discrete limit}.
\end{proof}

\section{Invariance principle for the spectrum of a random Schr\"odinger operator}\label{sec:schroedinger}

In the recent paper~\cite{AllezChouk2015} the random Schr\"odinger operator with white noise potential defined formally by
$$
\mathscr H=-\Delta+ (\xi+\infty)
$$
was constructed for the first time. It was shown that it is the limit (in resolvent sense) of the sequence of operators $\mathscr H^\delta=-\Delta+\xi_\delta+c_\delta$, $\delta > 0$, where $\xi_\delta$ is a mollification of the white noise and $c_\delta =\frac{1}{2\pi}\log(\frac{1}{\delta})+O(1)$ is the diverging constant appearing in~\eqref{eq:c-delta}. Moreover, it was shown that the operator $\mathscr H$ has a compact resolvent (in $L^2(\mathbb T^2)$) and a pure point spectrum 
$$
\sigma(\mathscr H)=\{\Lambda_1\leq\Lambda_2\leq\dots\leq\Lambda_k\leq \dots \},
$$ 
where $\Lambda_k\to+\infty$ for $k\to\infty$.

Now consider the operator $\mathscr H_N$ defined on the periodic lattice $\mathbb Z_N^2$ by  
$$
\mathscr H_Ne(i)=-(\Delta_{\mathrm{rw}}e)(i)+\eps e(i)\eta_N(i), \qquad i\in\mathbb Z_N^2,
$$ 
where $\eta_N$ is as in the previous sections. Given the results we derived so far it is natural to expect that the spectrum of $\mathscr H_N$ converges to that of $\mathscr H$, at least when suitably rescaled and recentered. To be precise we are interested in the convergence at the bottom of the spectrum of $\mathscr H_N$. That is, we fix $k\in\mathbb N$ and let $N\gg k$, we denote by $\Lambda_1^N\leq\dots\leq\Lambda_k^N$ the $k$ lowest eigenvalues of the operator $\mathscr H_N$, and our aim is to show the joint convergence of $(\Lambda_1^N,\dots, \Lambda_k^N)$. As usual in spectral analysis instead of studying this convergence directly we will prove that the rescaled resolvent of the operator $\mathscr H_N$  satisfies a central limit theorem which implies in particular the central limit theorem for the eigenvalues. Indeed, let us start by observing that if $e^k_N$ is an eigenfunction with eigenvalue $\Lambda^N_k$ and $\tilde e_N^k(i)=e_N^k(i/\eps)$ for $i\in\mathbb T^2_N$, then $\mathcal E_N\tilde e_N^k$  is an eigenfunction of the operator 
$$
\tilde{\mathscr H_N}f=-\Delta_{\mathrm{rw}}^N \PC_N f+\Pi_N(\PC_N (f)\xi_N),\qquad f\in L^2(\mathbb T^2),
$$
 with eigenvalue $\varepsilon^{-2}\Lambda_k^N$, where we recall that $\PC_N f = \CF^{-1} (\1_{(-N/2,N/2)^2} \CF f)$. Now the convergence of the eigenvalues of this self adjoint operator is implied by the convergence of its resolvent operator $(z+\tilde{\mathscr H_N})^{-1}$ for $z\in i\mathbb R \setminus \{0\}$. But as we have seen previously such a convergence can only be expected to hold after a suitable renormalization, and taking this into account we should study the operator $\tilde{\mathscr H_N}+c_N \mathcal{P}_N$ instead of $\tilde{\mathscr H_N}$. To prove the convergence of its resolvent (in the operator sense) it suffices to show that $g_N=(z+\tilde{\mathscr H_N}+ c_N \mathcal{P}_N)^{-1}f$ converges to $(z+\mathscr H)^{-1}f$ uniformly in $f\in L^2(\mathbb T^2)$ with $\|f\|_{L^2}=1$. For that purpose let us start by observing that $\PC_N g_N$ satisfies the equation 
$$
(z-\Delta_{\mathrm{rw}}^N)\PC_N g_N=\PC_N f - \Pi_N(\PC_N (g_N) \xi_N) - c_N \PC_N g_N,
$$  
whereas $(1- \PC_N) g_N = z^{-1} (1-\PC_N) f$. From here the convergence of $(g_N)$ can be established in the same manner as the convergence in Proposition~\ref{prop:deterministic discrete limit} (see~\cite{AllezChouk2015}, Proposition~4.13 and Lemma~4.15 for details on the resolvent equation in the paracontrolled framework), and in that way we obtain the following result.

\begin{proposition}
   Make assumptions (H$_{\mathrm{rw}}$) and (H$_{\mathrm{mart}}$). Let $\alpha \in (2 / 3, 1-2/p)$ and $z\in i\mathbb R \setminus\{0\}$. Then 
  the resolvent operator $(z+c_N+\tilde{\mathscr H_N})^{-1}$ converges in distribution in the space of bounded operators $L(L^2,H^{\alpha})$ to $(z+\mathscr H)^{-1}$.  
  \end{proposition}
As previously discussed this immediately yields the following corollary for the eigenvalues.

\begin{theorem}\label{thm:schroedinger}
Fix $k\in\mathbb N$ and let $N \gg k$. Let $\Lambda_1^N\leq \dots \leq\Lambda_k^N$ be the $k$ smallest eigenvalues of the operator $\mathscr H_N$ and $\Lambda_1 \le \dots \le \Lambda_k$ those of $\mathscr H$. Then for $N \to \infty$ the following convergence holds in distribution:
\[
   \varepsilon^{-2}(\Lambda_1^N,\dots,\Lambda_k^N)+c_N(1,\dots,1)\Rightarrow (\Lambda_1,\dots,\Lambda_k).
\]
\end{theorem}

\appendix

\section{A criterion for the weak convergence of Markov processes}

\begin{lemma}[\cite{Ethier1986}, Theorem 2.11 in Chapter 4]\label{lem:EK}
Let $E$ and $(E_N)_{N \in \N}$ be metric spaces such that $E$ is compact and separable and assume that for all $N$ we are given a measurable map $\psi_N \colon E_N \to E$ and a semigroup $(P_N(t))_{t \in [0,T]}$ of a Markov process $Y_N$ on $E_N$, such that $X_N = \psi_N (Y_N)$ has sample paths in $D([0,T],E)$. Assume also that there exists a Feller semigroup $(P(t))_{t \in [0,T]}$ such that 
\[
  \lim_{N \rightarrow 0} \|P_N(t)\pi_Nf - \pi_N P(t) f\|_{L^{\infty}} = 0
\]  
for every $f\in C(E, \R)$, where $\pi_N \colon L^{\infty}(E)\to L^{\infty}(E_N)$ is defined by the relation $\pi_N f (x) =f(\psi_N(x))$, $x \in E_N$. Then if $X_N(0)$ has a limiting probability distribution $\nu$ on $E$, the process $(X_N)$ converges in distribution in $D([0,T],E)$ to the Markov process $X$ starting at $\nu$ with semigroup $(P(t))_{t \in [0,T]}$.
\end{lemma}

\bibliographystyle{alpha}
\bibliography{bib}

\begin{bibdiv}
\begin{biblist}

\bib{AllezChouk2015}{article}{
      author={Allez, Romain},
      author={Chouk, Khalil},
       title={The continuous {A}nderson {H}amiltonian in dimension two},
        date={2015},
     journal={preprint arXiv:1511.02718},
}

\bib{Bahouri2011}{book}{
      author={Bahouri, Hajer},
      author={Chemin, Jean-Yves},
      author={Danchin, Raphael},
       title={{Fourier analysis and nonlinear partial differential equations}},
   publisher={Springer},
        date={2011},
}

\bib{Bailleul2016}{article}{
      author={Bailleul, Isma{\"e}l},
      author={Bernicot, Frederic},
       title={Heat semigroup and singular {PDE}s},
        date={2016},
     journal={Journal of Functional Analysis},
      volume={270},
      number={9},
       pages={3344\ndash 3452},
}

\bib{Bailleul2015}{article}{
      author={Bailleul, Isma{\"e}l},
      author={Bernicot, Frederic},
      author={Frey, Dorothee},
       title={Higher order paracontrolled calculus and 3d-{PAM} equation},
        date={2015},
     journal={arXiv preprint arXiv:1506.08773},
}

\bib{Bony1981}{article}{
      author={Bony, Jean-Michel},
       title={{Calcul symbolique et propagation des singularites pour les
  {\'e}quations aux d{\'e}riv{\'e}es partielles non lin{\'e}aires}},
        date={1981},
     journal={Ann. Sci. {\'E}c. Norm. Sup{\'e}r. (4)},
      volume={14},
       pages={209\ndash 246},
}

\bib{Brown1971}{article}{
      author={Brown, Bruce~M.},
       title={Martingale central limit theorems},
        date={1971},
     journal={Ann. Math. Statist.},
      volume={42},
      number={1},
       pages={59\ndash 66},
}

\bib{Bruned2015}{thesis}{
      author={Bruned, Yvain},
       title={Singular {KPZ} type equations},
        type={Ph.D. Thesis},
        date={2015},
}

\bib{CannizzaroChouk2015}{article}{
      author={Cannizzaro, Giuseppe},
      author={Chouk, Khalil},
       title={Multidimensional {SDE}s with singular drift and universal
  construction of the polymer measure with white noise potential},
        date={2015},
     journal={arXiv preprint arXiv:1501.04751},
}

\bib{Caravenna2013}{article}{
      author={Caravenna, Francesco},
      author={Sun, Rongfeng},
      author={Zygouras, Nikos},
       title={Polynomial chaos and scaling limits of disordered systems},
        date={2013},
     journal={arXiv preprint arXiv:1312.3357},
}

\bib{CarmonaMolchanov1994}{book}{
      author={Carmona, Ren{\'e}},
      author={Molchanov, Stanislav~A},
       title={Parabolic {A}nderson problem and intermittency},
   publisher={American Mathematical Soc.},
        date={1994},
      volume={518},
}

\bib{CatellierChouk2013}{article}{
      author={Catellier, R{\'e}mi},
      author={Chouk, Khalil},
       title={Paracontrolled distributions and the 3-dimensional stochastic
  quantization equation},
        date={2013},
     journal={arXiv preprint arXiv:1310.6869},
}

\bib{ChandraShen}{article}{
      author={Chandra, Ajay},
      author={Shen, Hao},
       title={Moment bounds for {SPDE}s with non-{G}aussian fields and
  application to the {W}ong-{Z}akai problem},
        date={2016},
     journal={arXiv preprint arXiv:1605.05683},
}

\bib{Ethier1986}{book}{
      author={Ethier, Stewart~N.},
      author={Kurtz, Thomas~G.},
       title={{Markov Processes: Characterization and Convergence}},
   publisher={WILEY},
        date={2005},
}

\bib{Friz2014}{book}{
      author={Friz, Peter~K},
      author={Hairer, Martin},
       title={A course on rough paths: with an introduction to regularity
  structures},
   publisher={Springer},
        date={2014},
}

\bib{Gubinelli2004}{article}{
      author={Gubinelli, Massimiliano},
       title={Controlling rough paths},
        date={2004},
     journal={Journal of Functional Analysis},
      volume={216},
      number={1},
       pages={86\ndash 140},
}

\bib{Gubinelli2013}{article}{
      author={Gubinelli, Massimiliano},
      author={Imkeller, Peter},
      author={Perkowski, Nicolas},
       title={Paracontrolled distributions and singular {PDE}s},
organization={Cambridge Univ Press},
        date={2015},
      volume={3},
       pages={e6},
}

\bib{Gubinelli2014KPZ}{article}{
      author={Gubinelli, Massimiliano},
      author={Perkowski, Nicolas},
       title={{KPZ} reloaded},
        date={2015},
     journal={arXiv preprint arXiv:1508.03877},
}

\bib{Gubinelli2014EBP}{article}{
      author={Gubinelli, Massimiliano},
      author={Perkowski, Nicolas},
       title={Lectures on singular stochastic {PDE}s},
        date={2015},
     journal={Ensaois Mat.},
      volume={29},
}

\bib{Hairer2011Rough}{article}{
      author={Hairer, Martin},
       title={Rough stochastic {PDE}s},
        date={2011},
     journal={Comm. Pure Appl. Math.},
      volume={64},
      number={11},
       pages={1547\ndash 1585},
}

\bib{hairer_solving_2011}{article}{
      author={Hairer, Martin},
       title={{Solving the KPZ equation}},
        date={2013},
     journal={Ann. Math.},
      volume={178},
      number={2},
       pages={559\ndash 664},
}

\bib{Hairer2014Regularity}{article}{
      author={Hairer, Martin},
       title={A theory of regularity structures},
        date={2014},
     journal={Invent. Math},
        note={DOI 10.1007/s00222-014-0505-4},
}

\bib{Hairer2016}{article}{
      author={Hairer, Martin},
       title={The motion of a random string},
        date={2016},
     journal={arXiv preprint arXiv:1605.02192},
}

\bib{HairerLabbe}{article}{
      author={Hairer, Martin},
      author={Labb{\'e}, Cyril},
       title={Multiplicative stochastic heat equations on the whole space},
        date={2015},
     journal={arXiv preprint arXiv:1504.07162},
}

\bib{Hairer_Maas_2012}{article}{
      author={Hairer, Martin},
      author={Maas, Jan},
       title={A spatial version of the {I}t\^o-{S}tratonovich correction},
        date={2012},
        ISSN={0091-1798},
     journal={Ann. Probab.},
      volume={40},
      number={4},
       pages={1675\ndash 1714},
         url={http://dx.doi.org/10.1214/11-AOP662},
      review={\MR{2978135}},
}

\bib{Hairer_Maas_2014}{article}{
      author={Hairer, Martin},
      author={Maas, Jan},
      author={Weber, Hendrik},
       title={Approximating rough stochastic {PDE}s},
        date={2014},
        ISSN={0010-3640},
     journal={Comm. Pure Appl. Math.},
      volume={67},
      number={5},
       pages={776\ndash 870},
         url={http://dx.doi.org/10.1002/cpa.21495},
      review={\MR{3179667}},
}

\bib{HairerMatetski2015}{article}{
      author={Hairer, Martin},
      author={Matetski, Konstantin},
       title={Discretisations of rough stochastic {PDE}s},
        date={2015},
     journal={arXiv preprint arXiv:1511.06937},
}

\bib{HairerPardoux}{article}{
      author={Hairer, Martin},
      author={Pardoux, {\'E}tienne},
       title={A wong-zakai theorem for stochastic {PDE}s},
        date={2015},
     journal={Journal of the Mathematical Society of Japan},
      volume={67},
      number={4},
       pages={1551\ndash 1604},
}

\bib{HairerQuastel}{article}{
      author={Hairer, Martin},
      author={Quastel, Jeremy},
       title={A class of growth models rescaling to {KPZ}},
        date={2015},
     journal={arXiv preprint arXiv:1512.07845},
}

\bib{HairerShenKPZ}{article}{
      author={Hairer, Martin},
      author={Shen, Hao},
       title={A central limit theorem for the {KPZ} equation},
        date={2015},
     journal={arXiv preprint arXiv:1507.01237},
}

\bib{Hairer2015Sine}{article}{
      author={Hairer, Martin},
      author={Shen, Hao},
       title={The dynamical sine-{G}ordon model},
        date={2016},
     journal={Comm. Math. Phys.},
      volume={341},
      number={3},
       pages={933\ndash 989},
}

\bib{HairerXu}{article}{
      author={Hairer, Martin},
      author={Xu, Weijun},
       title={Large scale behaviour of 3{D} phase coexistence models},
        date={2016},
     journal={arXiv preprint arXiv:1601.05138},
}

\bib{Hoshino2015}{article}{
      author={Hoshino, Masato},
       title={{KPZ} equation with fractional derivatives of white noise},
        date={2016},
     journal={arXiv preprint arXiv:1602.04570},
}

\bib{Hoshino2016}{article}{
      author={Hoshino, Masato},
       title={Paracontrolled calculus and {F}unaki-{Q}uastel approximation for
  the {KPZ} equation},
        date={2016},
     journal={arXiv preprint arXiv:1605.02624},
}

\bib{Janson1997}{book}{
      author={Janson, Svante},
       title={{Gaussian {H}ilbert spaces}},
      series={{Cambridge Tracts in Mathematics}},
   publisher={Cambridge University Press},
     address={Cambridge},
        date={1997},
      volume={129},
        ISBN={0-521-56128-0},
         url={http://dx.doi.org/10.1017/CBO9780511526169},
      review={\MR{1474726 (99f:60082)}},
}

\bib{Koenig2016}{book}{
      author={K\"{o}nig, Wolfgang},
       title={The parabolic anderson model. random walk in random potential},
   publisher={Birkh\"auser},
        date={2016},
}

\bib{Kupiainen2015}{article}{
      author={Kupiainen, Antti},
       title={Renormalization group and stochastic {PDE}s},
organization={Springer},
        date={2016},
     journal={Annales Henri Poincar{\'e}},
      volume={17},
      number={3},
       pages={497\ndash 535},
}

\bib{Kupiainen2016}{article}{
      author={Kupiainen, Antti},
      author={Marcozzi, Matteo},
       title={Renormalization of generalized {KPZ} equation},
        date={2016},
     journal={arXiv preprint arXiv:1604.08712},
}

\bib{Mossel2010}{article}{
      author={Mossel, Elchanan},
      author={O'Donnell, Ryan},
      author={Oleszkiewicz, Krzysztof},
       title={Noise stability of functions with low influences: Invariance and
  optimality},
        date={2010},
     journal={Ann. Math.},
      volume={171},
       pages={295\ndash 341},
}

\bib{Mourrat2014}{article}{
      author={Mourrat, Jean-Christophe},
      author={Weber, Hendrik},
       title={Convergence of the two-dimensional dynamic {I}sing-{K}ac model to
  $\phi^{4}_2$},
        date={2014},
     journal={arXiv preprint arXiv:1410.1179},
}

\bib{ProemelTrabs2015}{article}{
      author={Pr{\"o}mel, David~J},
      author={Trabs, Mathias},
       title={Rough differential equations driven by signals in {B}esov
  spaces},
        date={2016},
     journal={Journal of Differential Equations},
      volume={260},
      number={6},
       pages={5202\ndash 5249},
}

\bib{Shen2016}{article}{
      author={Shen, Hao},
      author={Weber, Hendrik},
       title={Glauber dynamics of 2{D} {K}ac-{B}lume-{C}apel model and their
  stochastic {PDE} limits},
        date={2016},
     journal={arXiv preprint arXiv:1608.06556},
}

\bib{ShenXu}{article}{
      author={Shen, Hao},
      author={Xu, Weijun},
       title={Weak universality of dynamical $\phi^4_3 $: non-{G}aussian
  noise},
        date={2016},
     journal={arXiv preprint arXiv:1601.05724},
}

\bib{Zhu2014Discretization}{article}{
      author={Zhu, Rongchan},
      author={Zhu, Xiangchan},
       title={Approximating three-dimensional {N}avier-{S}tokes equations
  driven by space-time white noise},
        date={2014},
     journal={arXiv preprint arXiv:1409.4864},
}

\bib{Zhu2015}{article}{
      author={Zhu, Rongchan},
      author={Zhu, Xiangchan},
       title={Lattice approximation to the dynamical $\phi_3^4$ model},
        date={2015},
     journal={arXiv preprint arXiv:1508.05613},
}

\end{biblist}
\end{bibdiv}

\end{document}